\documentclass[%
 reprint,
 amsmath,
 amssymb,
 aps,
 pre,
]{revtex4-1}

\usepackage{graphicx}
\usepackage{dcolumn}
\usepackage{bm}
\usepackage{hyperref}

\usepackage{xcolor}
\usepackage{amsthm}

\newtheorem{proposition}{Proposition}

\DeclareMathOperator{\sign}{sgn}
\DeclareMathOperator{\sss}{s}
\DeclareMathOperator{\ccc}{c}
\newcommand{\abg}{{\alpha,\beta,\gamma}}

\begin{document}

\preprint{APS/123-QED}

\title{Cutting and shuffling a hemisphere:\\non-orthogonal axes}

\author{Thomas F. Lynn}
\affiliation{Department of Engineering Sciences and Applied Mathematics,\\ Northwestern University, Evanston, Illinois 60208, USA}

\author{Lachlan D. Smith}
\affiliation{Department of Chemical and Biological Engineering,\\ Northwestern University, Evanston, Illinois 60208, USA}
\altaffiliation[Now at the ]{School of Mathematics and Statistics, The University of Sydney, Sydney, NSW 2006, Australia}

\author{Julio M. Ottino}
\affiliation{Department of Chemical and Biological Engineering, Department of Mechanical Engineering,\\ Northwestern University, Evanston, Illinois 60208, USA}
\affiliation{The Northwestern Institute on Complex Systems (NICO),\\ Northwestern University, Evanston, Illinois 60208, USA}

\author{Paul B. Umbanhowar}
\affiliation{Department of Mechanical Engineering,\\ Northwestern University, Evanston, Illinois 60208, USA}

\author{Richard M. Lueptow}
\email{r-lueptow@northwestern.edu}
\affiliation{Department of Mechanical Engineering, Department of Chemical and Biological Engineering,\\ Northwestern University, Evanston, Illinois 60208, USA}
\affiliation{The Northwestern Institute on Complex Systems (NICO),\\ Northwestern University, Evanston, Illinois 60208, USA}

\date{\today}

\begin{abstract}
We examine the dynamics of cutting-and-shuffling a hemispherical shell driven by alternate rotation about two horizontal axes using the framework of piecewise isometry (PWI) theory. Previous restrictions on how the domain is cut-and-shuffled are relaxed to allow for non-orthogonal rotation axes, adding a new degree of freedom to the PWI. A new computational method for efficiently executing the cutting-and-shuffling using parallel processing allows for extensive parameter sweeps and investigations of mixing protocols that produce a low degree of mixing. Non-orthogonal rotation axes break some of the symmetries that produce poor mixing with orthogonal axes and increase the overall degree of mixing on average. Arnold tongues arising from rational ratios of rotation angles and their intersections, as in the orthogonal axes case, are responsible for many protocols with low degrees of mixing in the non-orthogonal-axes parameter space. Arnold tongue intersections along a fundamental symmetry plane of the system reveal a new and unexpected class of protocols whose dynamics are periodic, with exceptional sets forming polygonal tilings of the hemispherical shell.
\end{abstract}

\maketitle



\section{Introduction} \label{sec:introduction}


Mixing of fluids by diffusion, chaotic advection, and turbulence has been well studied \cite{Wiggins2004, Ottino1988a}.
%
Cutting-and-shuffling is a mixing mechanism that is far less understood, but is particularly relevant to systems with flow discontinuities, such as granular materials \cite{Sturman2012, Juarez2010, Juarez2012, Park2017, Smith2017, Umbanhowar2013}, valved fluid flow \cite{Jones1988, Smith2016a}, thrust faults in geology \cite{Boyer1982, Bell1983, Butler1982}, and, of course, the typical example of mixing a deck of cards \cite{Golomb1961, Aldous1986,  Trefethen2000}. 
In one dimension, cutting-and-shuffling is described mathematically by interval exchange transforms (IETs) \cite{Katok1980, Hmili2010, Keane1977, Masur1982, Avila2007, Krotter2012, Novak2009, Yu2016, Wang2018, Keane1975, Veech1978, Viana2006} which naturally extend to higher dimensions under the framework of \textit{piecewise isometries} (PWIs). PWIs, which cut an object into pieces and spatially rearrange them to form the original shape, can produce complex dynamics despite their relative simplicity \cite{Sturman2008, Ashwin1997a, Scott2001,  Goetz2000, Ashwin2002b, Goetz2001, Goetz2004, Fu2008, Kahng2009}. There are several, though somewhat similar, definitions of PWIs \cite{Ashwin2002b, Goetz2001, Goetz2004, Fu2008, Kahng2009, Park2017}. A piecewise isometry (PWI) \(M: S \to S\) is a map on a domain \(S\) such that, for some partition of \(S\) into a finite number \(N\) of closed \footnote{Most definitions \cite{Ashwin2002b, Goetz2001, Goetz2004, Fu2008, Kahng2009}, treat atoms as open sets and PWI actions are undefined (instead of multi-valued) on atom boundaries.}, mutually disjoint (up to their boundaries) partition elements \(\{P_i\}_{i = 1}^N\) (termed \textit{atoms}), the action of \(M\) is a Euclidean isometry (length preserving, e.g. rotation, translation, reflection) on each \(P_i\). A PWI is invertible if the mapped atoms, \(\{MP_i\}_{i = 1}^N\), are also mutually disjoint (again, up to their boundaries). Overlapping atom boundaries are treated as members of both adjacent atoms, resulting in a map that is multi-valued on atom boundaries and introduces complications that ultimately have no bearing on the measurements in this paper. A PWI is orientation preserving if there are no reflections.

We investigate a specific family of invertible, orientation preserving PWIs on a hemispherical shell \cite{Juarez2010, Juarez2012, Sturman2008, Park2016, Park2017, Smith2017, Christov2014, Meier2007}. The hemispherical shell has become a prototypical PWI for its relation to a half-filled spherical tumbler of granular particles rotated sequentially about two different horizontal axes (called the Bi-axial Spherical Tumbler, or BST). The stripped down version of this system is a PWI on a hemispherical shell.

Previous studies of mixing by PWIs on hemispheres \cite{Sturman2008, Juarez2010, Christov2014, Meier2007} or hemispherical shells \cite{Sturman2008, Park2016, Park2017, Smith2017} focused almost exclusively on the case where the two rotation axes are orthogonal with the single apparent exception of Juarez et al.\ \cite{Juarez2012} where a general mixing metric for non-orthogonal axis protocols was investigated using this same PWI. An investigation of a related system for a non-hemispherical spherical cap (i.e., less-than-half-filled hemisphere), which induces shear along the axial direction (\(\partial \bm{u} / \partial \bm{a}\)) and is therefore no longer a PWI, was recently carried out by Smith et al.\ \cite{Smith2017d}.

This paper uses the PWI formulation to examine mixing on a hemispherical shell when the restriction of orthogonal axes is relaxed, as described in Section~\ref{sec:pwi-def} (the PWI studied here can be expanded to a full sphere, as shown in Appendix~\ref{sec:extended-pwi}).
First, we describe a highly efficient approach for computing the ``exceptional set'' associated with PWIs in Section~\ref{sec:tildeE} and determine its areal coverage, which is correlated with the degree of mixing \cite{Park2017}. Second, we explore symmetry-breaking and other effects that occur when rotations of the hemispherical shell occur about non-orthogonal axes, instead of orthogonal axes, in Section~\ref{sec:non-orthog-E}. Finally, in Section~\ref{sec:phase_space} we examine the position of resonant structures (protocols with minimal coverage) within the non-orthogonal parameter space as the angle between axes changes. Some resonances have polygonal non-mixing regions and some are polygonal tilings of the hemispherical shell. Section~\ref{sec:conclusions} presents our conclusions.

\section{PWI mapping for non-orthogonal axes} \label{sec:pwi-def}

\begin{figure}
  \includegraphics[width = 0.45\textwidth]{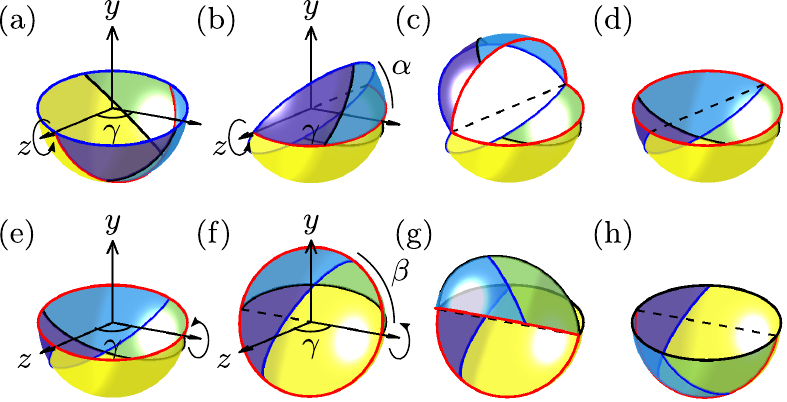}
  \caption{Demonstration of the rotations that define the PWI, here \(\alpha = 57^{\circ}\), \(\beta = 57^{\circ}\), and \(\gamma = 120^{\circ}\). (a) The initial condition of the hemispherical shell (HS). (b) Rotation about the first axis (\(z\)-axis) by \(\alpha\).
  (c) Separation of the portion above the equator and beginning of the rotation to reconstruct the HS.
  (d) Re-formed HS.
  (e-h) Repeat of (a-d) about the second horizontal axis.}
  \label{fig:pwi-demo}
\end{figure}

The procedure defining the PWI on a hemispherical shell (HS) as well as the coordinate system used in this paper is shown schematically in Fig.~\ref{fig:pwi-demo}. Begin by rotating the hemispherical shell about the first horizontal axis by angle \(\alpha\) [Fig.~\ref{fig:pwi-demo}(a) to (b)], rotating the section that passes above the equator by an additional \(180^\circ\) [Fig.~\ref{fig:pwi-demo}(c) and (d)], which provides the cutting action. Next, rotate about the second axis by angle \(\beta\) [Fig.~\ref{fig:pwi-demo}(e) to (f)], again rotating the section that passes above the equator by an additional \(180^\circ\) [Fig.~\ref{fig:pwi-demo}(g) and (h)]. The angle between the rotation axes is \(\gamma\). Thus, we consider the PWI mapping on a hemispherical shell (HS), \(M_\abg: S \to S\), which maps the lower hemispherical unit shell, \(S = \{(x,y,z): x^2 + y^2 + z^2 = 1, y \le 0\}\), to itself.
Note that the procedure here differs from some previous definitions \cite{Park2016, Park2017, Smith2017} in that both rotations are counterclockwise (this merely results in a left-right mirroring of the PWI compared to other work). More importantly, the rotation axes are not necessarily orthogonal to one another, as in the degree of mixing study by Juarez et al.\ \cite{Juarez2012}.

Together, the two rotations shown in Fig.~\ref{fig:pwi-demo} make up the PWI mapping \( M_\abg\); the ordered triple of control parameters \((\abg)\) is termed the mapping \textit{protocol}. The mapping can be written as the composition of the two modular rotations about horizontal axes \(\bm{a}_1\) and \(\bm{a}_2\) (applied right to left),
\begin{equation}
  M_\abg = \tilde{M}_\beta^{\bm{a}_2} \tilde{M}_\alpha^{\bm{a}_1}, \label{eq:mapping}
\end{equation}
where \(\tilde{M}^{a}_\theta\) represents a rotation about axis \(a\) by angle \(\theta\) with the condition that points crossing the equator are rotated additionally by \(180^{\circ}\) to reconstruct the HS, as shown schematically in Fig.~\ref{fig:pwi-demo}. Here, axes \(\bm{a}_1\) and \(\bm{a}_2\) lie in the equatorial \(xz\)-plane with angle \(\gamma\) between them.

The physical description outlined in Fig.~\ref{fig:pwi-demo} helps to visualize the cutting and shuffling of the HS and relate it to the physical system, but the composition of these actions alone defines the PWI mapping. Figure~\ref{fig:pwi_action} shows the PWI \(M_{57^{\circ},57^{\circ},120^{\circ}}\) viewed from below (along the negative \(y\)-axis) as an example. The hemispherical surface is split into four \textit{atoms}, the closed regions labeled \(P_1, P_2, P_3, P_4 \) in Fig.~\ref{fig:pwi_action}(a), and rearranged to reconstruct the HS as shown in Fig.~\ref{fig:pwi_action}(b). Thus, the combination of step-wise cuts and rotations in Fig.~\ref{fig:pwi-demo} is formally equivalent to the rearrangement of \(P_1, P_2, P_3, P_4 \) between Fig.~\ref{fig:pwi_action}(a) and Fig.~\ref{fig:pwi_action}(b) as a single action.

%

The domain is split along curved cutting lines (great circle arcs) which represent discontinuities in the map, shown in Fig.~\ref{fig:pwi_action}(a) as the red arc \(\mathcal{D}_1\)  and the two black arcs that make up \(\mathcal{D}_2\). The atoms \(P_1,P_2,P_3,P_4\) are rearranged and combined along the cutting lines of the inverse map, shown in Fig.~\ref{fig:pwi_action}(b) as the blue and red arcs. Note that the black cutting lines \(\mathcal{D}_2\) in Fig.~\ref{fig:pwi_action}(a) become the equatorial edge of the HS, which is called \(\partial S\), in Fig.~\ref{fig:pwi_action}(b), and the blue equatorial edge in Fig.~\ref{fig:pwi_action}(a) becomes the blue cutting lines in Fig.~\ref{fig:pwi_action}(b). The union of cutting lines \(\mathcal{D}_1\) and \(\mathcal{D}_2\) is the set of discontinuities \(\mathcal{D}\), also called the unstable set \cite{Scott2003}, which is formally defined as the union of all intersections of the closed partition elements of \(S\), \(P_i\), such that \(\mathcal{D} = \bigcup_{i\neq j} P_i \cap P_j\) \cite{Park2017} \footnote{It is sometimes convenient to include the edge of the HS, \(\partial S\), with \(\mathcal{D}\), but this is not necessary when multiple iterations of the PWI are considered since \(\partial S = M_\abg \mathcal{D}_2\).}. 

\begin{figure}
  \includegraphics[width = 0.5\textwidth]{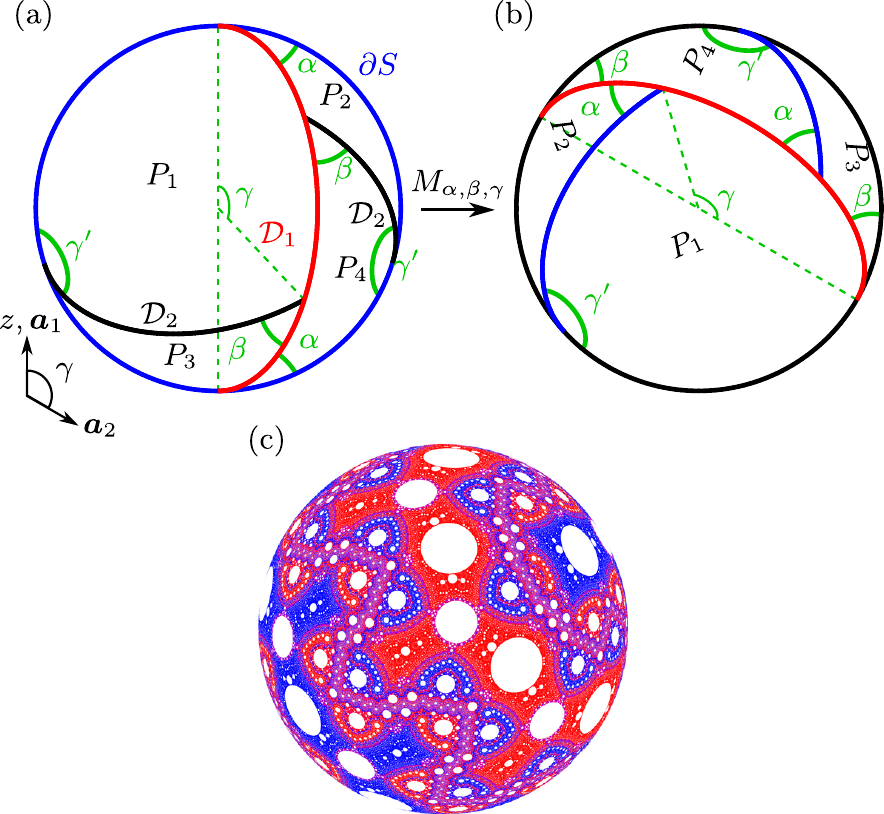}
  \caption{The PWI mapping \(M_{57^{\circ},57^{\circ},120^{\circ}}\) (shown from the negative \(y\)-axis) (a) cuts the domain into four atoms, \(P_1, P_2, P_3, P_4 \), and (b) rearranges them according to the rotation procedure in Sec. \ref{sec:pwi-def}. Cutting lines \(\mathcal{D}_1\) and \(\mathcal{D}_2\) are the red and black arcs in (a). Angles \(\alpha, \beta, \gamma\) specify the size and orientation of the atoms. Angle \(\gamma'\) is the spherical angle between cutting planes.
  (c) The union of cutting lines (red and blue) when \(M_\abg\) is applied 20,000 times.
  }
  \label{fig:pwi_action}
\end{figure}

It is more difficult to visualize the angles \(\alpha\), \(\beta\), and \(\gamma\) in Fig.~\ref{fig:pwi_action} than in Fig.~\ref{fig:pwi-demo}, but these angles are still present despite being distorted by the view from the bottom of the hemispherical shell. For example, \(\beta\) is the angle between the black cutting line \(\mathcal{D}_2\) and the red cutting line \(\mathcal{D}_1\) that form two edges of atoms \(P_3\) and \(P_4\). The area of each atom (a unit hemisphere having a total area of \(2\pi\)) is the sum of the angles in the spherical triangle that specifies the atom minus \(\pi\) such that the surface area \(A_i\) of \(P_i\) is

\begin{align}
A_1 &= \pi - \alpha - \beta + \gamma', \label{eq:area_p1}\\
A_2 &= \pi + \alpha - \beta - \gamma', \label{eq:area_p2}\\
A_3 &= \pi - \alpha + \beta - \gamma', \label{eq:area_p3}\\
A_4 &= - \pi + \alpha + \beta + \gamma',\label{eq:area_p4}
\end{align}
where the spherical angle \(\gamma'\) is the angle between the cutting planes forming cutting lines \(\mathcal{D}_1\) and \(\mathcal{D}_2\), given by
\begin{equation}
\gamma' = \arccos[ \cos(\gamma) \sin(\alpha) \sin(\beta) - \cos(\alpha) \cos(\beta) ].
\label{eq:gammaprime}
\end{equation}

Applying the \(M_{57^{\circ}, 57^{\circ}, 120^{\circ}}\) map \(2\times 10^4\) times and recording the positions of the cutting lines at the end of each iteration while keeping the blue and red coloring of each of the cutting lines, generates the pattern shown in Fig.~\ref{fig:pwi_action}(c) where red and blue are used for the first and second cuts of the protocol, respectively \footnote{Since \(M_\abg \mathcal{D}_2 = \partial S\), the union of all iterates of \(\mathcal{D}_2\) and \(\partial S\) coincide. Therefore, we need only show the blue lines in Fig.~\ref{fig:pwi_action}. We only show blue lines for multiple iterations.}. The overall structure has white non-mixing islands (uncut regions), termed \textit{cells} \cite{Park2017, Goetz2001, Goetz2000}, spread throughout the domain. In addition, while it is not obvious from Fig.~\ref{fig:pwi_action}(c), the structure is fractal \cite{Park2017}.

As the number of iterations, \(N\), approaches infinity, the structure formed by the cutting lines (top row of Fig.~\ref{fig:increasing-iters}) reveals the \textit{singular set}, \(E\), which is the union of all pre- and post-images of the discontinuities \(\mathcal{D}\) \cite{Fu2008, Park2016, Ashwin2005, Kahng2009},
\begin{equation}
 E = \bigcup^\infty_{n=-\infty} M^n_\abg \mathcal{D}.
\end{equation}
The closure of \(E\) is the \textit{exceptional set}, \(\bar{E}\), containing \(E\) and its limit points. The exceptional set has been referred to as the ``skeleton'' of mixing because of its structural role in specifying mixing and non-mixing regions \cite{Juarez2010, Smith2017, Zaman2018}. Visually, \(E\) approximates \(\bar{E}\) when cutting lines are given non-zero thickness.

A Hamiltonian system for a kicked harmonic oscillator that exhibits similar behavior \cite{Scott2001,Scott2003} provides strong numerical evidence that cells form a circle packing of the domain that is not dense, suggesting that \(\bar{E}\) is a fat fractal \cite{Farmer1983, Grebogi1985} (i.e.\ has non-zero measure) for almost all protocols \cite{Park2017}. Since \(E\) is a countable collection of zero-measure lines, the non-zero measure of \(\bar{E}\) comes entirely from the boundary points of \(\bar{E}\), \(\bar{E}\setminus E\) where ``\(\setminus\)'' denotes \(E\) is removed from \(\bar{E}\).

An example of how the cutting line structure develops for increasing iterations of the \((90^{\circ},90^{\circ},60^{\circ})\) map is shown in Fig.~\ref{fig:increasing-iters}(a). Note that the prominent circular cells become visible after only 20 iterations, though at this point they are 14-sided (irregular) polygons with each side formed by a cutting line. Successive iterations of the PWI rotates the 14-sided polygons about a central axis, trimming off its corners and leading to a truly circular cell as \(N \rightarrow \infty\). The set of cells, or non-mixing islands, is called the stable or regular set \cite{Scott2003,Ashwin2018}. The stable set is the complement of \(\bar{E}\) and contains all non-mixing islands. Typically, the stable set has structure at all scales; for this protocol it forms intricate, but incomplete, circle packing of the HS \cite{Park2017, Scott2003}. Cells are the maximally open neighborhoods around periodic points of the domain that contain points with the same \textit{periodic itinerary} [i.e., a symbolic representation of an orbit by way of the atom labels (1,2,3,4) in Fig.~\ref{fig:pwi_action}(a)] \cite{Bruin2003, Smith2017}. The complement of \(\bar{E}\), evident in Fig. \ref{fig:pwi_action}(c) and Fig.~\ref{fig:increasing-iters}(a) at \(2\times10^4\) iterations as the white portion of the domain, is the remainder of the HS that is never cut.

The relationship of the exceptional set to mixing can be understood by iterating the PWI mapping to scramble a scalar field on the HS. The mixing of a scalar field---an initial condition with continuously varied shading---due to the \((90^{\circ},90^{\circ},60^{\circ})\) map as it is iterated is shown on the left of Fig.~\ref{fig:increasing-iters}(b). With each iteration, regions of the domain are cut and shuffled, while other regions remain uncut. By 20 iterations, the large non-mixing circular cells have taken shape. These circular cells follow periodic paths through the domain as the mapping is iterated, but also undergo internal rotation about a central elliptic periodic point within the cell \cite{Smith2017, Park2016, Park2017}.
This internal rotation is evident in the large lower left cell (as well as other period-2 cells in the map), which has a light color on the left at 20 iterations and the same light color on the right at 200 iterations, at the top at \(2\times 10^3\), and back on the left at \(2\times 10^4\) iterations. Cells are often circular in shape due to irrational internal rotation about a central elliptic periodic point \cite{Smith2017}. However, a cell that rotates about its periodic center by a rational fraction of \(\pi\) will be polygonal.

The open circular cells outside of \(E\) correspond to unmixed circular regions in iterates of the scalar field, while dense regions of cutting lines in \(E\) correspond to well-mixed regions in iterates of the scalar field, which appear gray (the average color of the initial condition).  The correspondence between \(E\) and mixing for non-orthogonal axes is consistent with previous results for orthogonal rotation axes (\(\gamma =90^{\circ}\)) \cite{Park2016}.

\section{Measuring the coverage of \(\bar{E}\)} \label{sec:tildeE}

One measure of the degree of mixing of the the domain, is the \textit{fractional coverage} of \(\bar{E}\), \(\Phi\) \cite{Park2017, Ashwin1997a}, which is inversely correlated with the Danckwerts' intensity of segregation, a measure of the degree of mixing for segregated initial conditions \cite{Park2017, Danckwerts1952}. Formally, \(\Phi\) is the Lebesgue measure (here a measurement of 2D area) of \(\bar{E}\) normalized by the area of the unit  hemispherical shell, \(\Phi = \mu_L(\bar{E})/2\pi\) \cite{Park2017}. Although cutting lines have zero measure, \(\bar{E}\) has non-zero measure for almost all protocols \cite{Scott2003, Park2017, Smith2017}
\footnote{We conjecture that the set of protocols for which \(\bar{E}\) has zero measure has zero measure in the protocol space. \(\bar{E}\) has zero measure if and only if it is a polygonal tiling, consisting of rational cellular rotation (countably infinite rationals) about a countable (possibly infinite) number of periodic points. For any fixed \(\gamma\), there are then a countably infinite number of polygonal tilings resulting in zero measure in the protocol space. The conjecture is false if there is some positive area region in the protocol space that results in the same polygonal tiling, requiring that internal rotation angles are constant in the positive area region. If internal rotation angles are analytic, this would imply internal rotation angles are constant everywhere, obviously false.}.

\begin{figure*}
  \includegraphics[width = 0.95\textwidth]{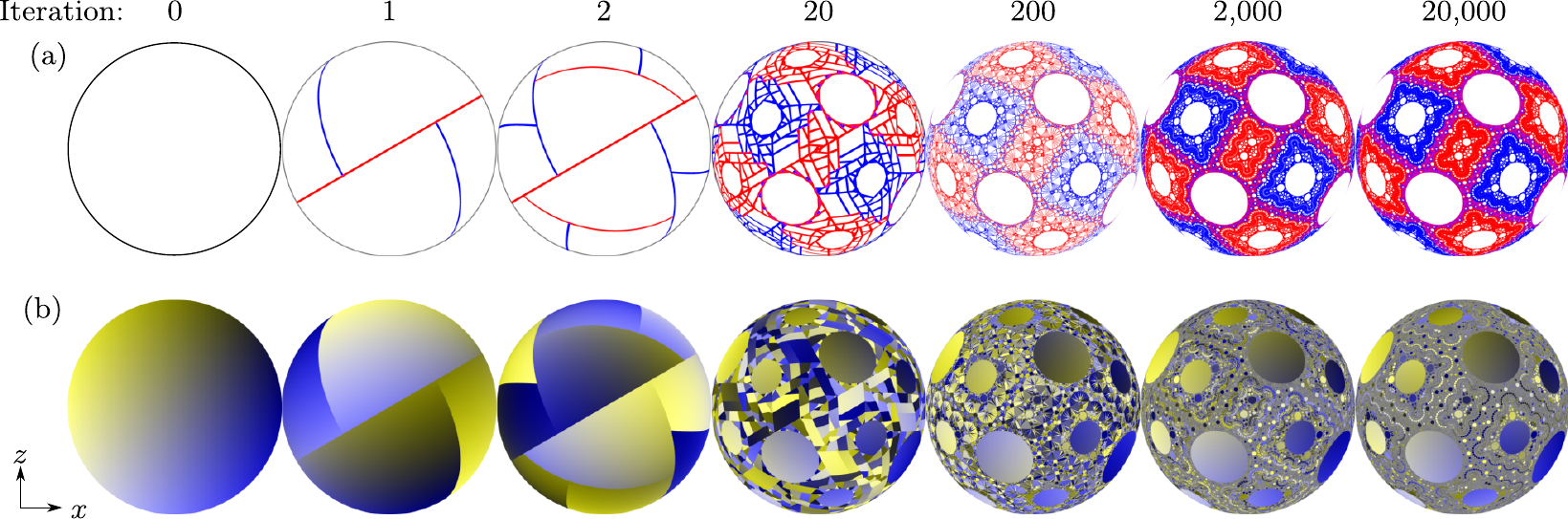}
  \caption{Iterating the \((90^{\circ},90^{\circ},60^{\circ})\) map. (a) The cutting line structure approaches the structure of the exceptional set as the number of iterations increases. Circular regions appear around periodic points. We use \(\varepsilon = 0.01\) for \(1 \le N \le 20\) and \(\varepsilon = 0.001\) for \(200\le N \le 2\times 10^4\). Larger values of \(\varepsilon\) are used for small numbers of iterations to give a visible thickness to the few cutting lines that are present. (b) With a colored initial condition (\(N=0\)), the mixed part of the domain for \(N = 2\times 10^4\) (gray region) closely matches regions in (a) that are densely cut, while circular unmixed regions correspond to open cells in (a).
  }
  \label{fig:increasing-iters}
\end{figure*}

Comparing the fractional coverage, \(\Phi\), of a large number of protocols is a computationally challenging task. To efficiently construct an approximation for the structure of \(E\) and subsequently obtain an estimate for \(\Phi\) as a function of protocol, we approximate \(\bar{E}\) by giving the cutting lines that make up \(E\) a finite width, \(\varepsilon > 0\), referred to as \(\mathcal{D}_\varepsilon = \{x \in B_\varepsilon(y), y\in \mathcal{D}\}\), i.e.\ all \(x\) within a geodesic \(\varepsilon\)-radius ball of any point \(y\) in \(\mathcal{D}\). We refer to this fattened \(E\) as \(E_\varepsilon = \{x \in B_\varepsilon(y), y \in E\}\), and the approximation of \(E_\varepsilon\) after \(N\) iterations of the PWI applied to \(\mathcal{D}_\varepsilon\) as \(\tilde{E}_{\varepsilon,N}\). Grebogi et al.\ \cite{Grebogi1985} utilize a similar fattening in determining the external dimension of fat fractals. Whether or not \(N\) is sufficient to completely cover \(\bar{E}\) with \(\tilde{E}_{\varepsilon,N}\) depends on the choice of \(\varepsilon\) and the specific PWI protocol. But this fattened exceptional set can completely cover \(E\) within a finite number of iterations since for every \(\varepsilon > 0\), there exists a finite \(N > 0\) such that 
\begin{equation}
  \label{eq:fat_coverage}
  E \subset \bar{E} \subset  \tilde{E}_{\varepsilon,N} = \bigcup^N_{n=0} M^n_\abg \mathcal{D}_\varepsilon  \subset E_\varepsilon.
\end{equation}
In other words, after a finite number of iterations, \(N\), all cells with radius smaller than \(\varepsilon\) are covered by \(\tilde{E}_{\varepsilon,N}\), thereby completing the approximation such that any further iterations do not add additional information about the location or size of \(\bar{E}\). We refer to \(\tilde{E}_{\varepsilon,N}\) as ``complete'' if \(N\) is at least sufficient to satisfy Eq.~\ref{eq:fat_coverage}. Appendix~\ref{sec:fat_coverage_proof} provides a proof of Eq.~\ref{eq:fat_coverage}, though the proof gives no information about how to find \(N\) \footnote{In fact, for any choice of \(N\) and \(\varepsilon\), there exist protocols such that \(\tilde{E}_{\varepsilon,N}\) is not complete. Take, for example, a single axis rotation with sufficiently small rotation angle. The number of iterations for completeness \(N\) apporaches \(\infty\) as the rotation angle approaches zero. In other words, there is no universal \(N\) that guarantees completeness. We conjecture that there is a universal bound away from these pathological/degenerate cases.}. In the limit as \(\varepsilon \rightarrow 0\) and \(N \rightarrow \infty\), the fat-lined fractional coverage of \(\tilde{E}\), \(\Phi_{\varepsilon,N}\), is equal to the fractional coverage of \(\bar{E}\), \(\Phi\). \(\Phi\) is expected to be positive for almost all protocols based on box-counting measurements of \(E\) \cite{Park2017}. Although \(E\) is formally the union of all pre- and post-images of \(\mathcal{D}\) under \(M_\abg\), due to symmetry in the PWI mapping, only the post-images (or the pre-images) of \(\mathcal{D}\) are required to construct all of \(E\) [see discussion in Appendix~\ref{sec:symmetries} around Eq.~\ref{eq:plus_minus_equiv} for additional details].

In previous studies, \(E\) was constructed by seeding points along (or near) \(\mathcal{D}\) and iterating their positions \cite{Park2016,Park2017}. This method has the advantages of providing a close approximation to the structure of \(E\) if seeded points in \(\mathcal{D}\) are sufficiently dense and \(\Phi\) can be directly measured using a box-counting method with equal-area boxes throughout the hemisphere. However, this Lagrangian, cutting-line-centered, box-counting method has two major shortcomings: (1) as \(\mathcal{D}\) is iterated, it splits into small segments and eventually individual seed points separate from one another, destroying all knowledge of the curves between separated points; (2) \(\mathcal{D}\) and its iterates are decoupled from the HS,  which means the entirety of \(E\) must be generated to investigating the structure of \(E\) in small regions of the HS, which results in significant memory usage and wasted computation, especially for large numbers of iterations.

The following Eulerian, fat-line, domain-centered method addresses the shortcomings of the cutting-line-centered box-counting method and optimizes generation of \(\tilde{E}\) by utilizing parallel computing. \(\tilde{E}\) can be computed by iterating the positions of a grid of tracer points using the PWI and labeling all points in the grid that fall within \(\mathcal{D}_{\varepsilon}\) at some iteration.

Using the fat-line method to compute \(\Phi\), \(\tilde{E}\) can be imagined as a sieve which is dusted with a uniform distribution of seed-points. The fraction of seed-points that do not fall through the sieve is the value of \(\Phi_{\varepsilon,N}\). The accuracy of the measurement is determined by the number of points used, but the value of \(\Phi_{\varepsilon,N}\) depends entirely on the sieve, \(\tilde{E}_{\varepsilon,N}\).

Since measuring \(\Phi_{\varepsilon,N}\) depends only on the fraction of grid points within \(\tilde{E}_{\varepsilon, N}\), computation can cease once a point's membership in \(\tilde{E}_{\varepsilon,N}\) is confirmed or \(N\) iterations are completed. The points used to poll \(\tilde{E}\) are selected from an equal-area distribution on the hemisphere to ensure an accurate measurement of area. In this paper, a Cartesian grid of polling points is projected onto the hemisphere using the Lambert azimuthal equal-area projection \cite{Snyder1987}, details of which are provided in Appendix~\ref{sec:lambert}.

The fat-line method avoids the break-up of cutting-lines as the number of iterations increases since cutting lines remain fixed and continuous. It also couples the domain and cutting lines, which allows selective computation of details of sub-regions without wasting computation time and memory on other regions. Since seeded points are independent from one another, the fat-line method is ``embarrassingly parallel'' \cite{Herlihy2012}. Therefore, it can be executed efficiently using a GPU architecture (here, NVIDIA's CUDA architecture) by assigning each polling-point its own GPU core. The increase in speed from the serial box-counting method for a comparable resolution is more than three orders of magnitude. 

Furthermore, the fat-line method is mostly independent of the resolution of the seed-point grid. Although more points increase the accuracy of the measurement of \(\Phi_{\varepsilon,N}\), the true value of \(\Phi_{\varepsilon,N}\) is wholly determined by \(\varepsilon\) and \(N\). Measurements of \(\Phi_{\varepsilon,N}\) obtained in this way are uncertain when the number of seed-points are low, but the general trends in \(\Phi_{\varepsilon,N}\) across protocols, discussed in Section~\ref{sec:phase_space}, remain unaffected by this uncertainty.

The problem of guaranteeing that \(\tilde{E}\) has been sufficiently resolved by \(N\) iterations is not unique to this method, and an \textit{a priori} method for determining \(N\) as a function of \(\varepsilon\) and protocol is not known to us. Previous box-counting methods relied on the number of iterations between visiting new boxes as a metric for determining completeness of an exceptional set and stopping the computation \cite{Park2017}. Although a similar stopping condition could be implemented using the grid of tracer points, computations are so fast on a parallel architecture that it is unnecessary to implement for almost all protocols except those asymptotically close to polygonal tilings where internal rotation of cells closely approaches a rational multiple of \(\pi\) (i.e., cells only approach circles as \(N \to \infty\)). Instead, typical values of \(N=2\times10^4\) iterations with \(\varepsilon = 1\times 10^{-4}\) are used with a \((2\times 10^3) \times (2\times 10^3)\) Cartesian grid of tracer points to generate images of \(\tilde{E}_{\varepsilon,N}\) for the protocols in this paper except when noted.

Since the fat-line method iterates the domain and not the cutting lines, other interesting features of the PWI can be easily measured. Mixing an initial condition such as that shown in Fig.~\ref{fig:increasing-iters}(b) becomes a trivial task of mapping a color from a tracer point's initial to its final position. When mixing initial conditions, as in Fig.~\ref{fig:increasing-iters}, it is more convenient to use the inverse mapping \(M_\abg^{-1}\) so that mixed conditions due to forward iterations end up on the well-aligned grid used to seed the HS (this is the \textit{Perron-Frobenius operator} for transforming scalar fields \cite{Lasota1994}).

\section{Effects of non-orthogonal axes \label{sec:non-orthog-E}}

\subsection{Changes in the exceptional set}

Using the approach described above, the effect of non-orthogonal axes on mixing and the structure of \(E\) can be explored. Starting from protocols with orthogonal axes (\(\gamma = 90^{\circ}\)), Fig.~\ref{fig:change-gamma} shows how \(\tilde{E}\) changes as the angle between axes increases to \(120^{\circ}\) for two different protocols. Consider first the \((90^{\circ},90^{\circ},90^{\circ})\) protocol in Fig.~\ref{fig:change-gamma}(a). Because of the rotation symmetry and orthogonal axes, this protocol results in a simple \textit{resonance}, or locally minimal mixing protocol that corresponds to a fully periodic structure and a polygonal tiling of the hemisphere \cite{Smith2017}. As \(\gamma\) increases from \(90^{\circ}\), \(\tilde{E}\) begins to fill in, losing its tiled appearance by Fig.~\ref{fig:change-gamma}(b) for the \((90^{\circ},90^{\circ},95^{\circ})\) protocol, effectively breaking the horizontal and vertical symmetries inherent to \((90^{\circ},90^{\circ},90^{\circ})\). The period-2 cell labeled \(A\) which rotates internally by \(120^{\circ}\) (a rational fraction of \(\pi\)) after every second iteration for \(\gamma = 90^{\circ}\), now rotates internally by \(119.49...^{\circ}\) (an irrational fraction of \(\pi\)) resulting in a circular cell. Within the \((90^{\circ},90^{\circ},95^{\circ})\) structure, Moir\'e-like patterns appear around the large period-2 cells which seem to be related to cutting lines tangent to the large cells. When \(\gamma\) is increased by another \(5^{\circ}\) to \((90^{\circ},90^{\circ},100^{\circ})\) in Fig.~\ref{fig:change-gamma}(c), smaller cells appear between the larger period-2 cells. At \(\gamma = 120^{\circ}\) in Fig.~\ref{fig:change-gamma}(d), the period-2 cells have significantly diminished in size and large period-4 cells, one of which is labeled \(B\), have appeared. It is evident that the fractional coverage of \(\tilde{E}\) increases as \(\gamma\) increases, and corresponding values of \(\Phi_{\varepsilon,N}\) are indicated in the figure.

Variation in \(\tilde{E}\) also occurs for the \((57^{\circ},57^{\circ},\gamma)\) protocols, but it is somewhat different. For orthogonal axes, \(\tilde{E}\) has relatively large cells (like period-3 cell \(C\)) surrounded by smaller cells, as shown in Fig.~\ref{fig:change-gamma}(e). When \(\gamma\) is increased to \(95^{\circ}\) in Fig.~\ref{fig:change-gamma}(f), the period-3 cells increase in size. All other cells have decreased in size and many smaller cells have appeared. Nearly all of the smaller cells disappear when \(\gamma\) increases to \(100^\circ\), leaving only six larger period-3 cells and ten small period-5 cells (four of these cells are barely visible near the edge of the HS). Note that \(2n\) period-\(n\) cells would be expected \cite{Christov2014}, \(n\) from a regular set and \(n\) from a conjugate set. When \(\gamma = 120^{\circ}\), the period-3 cells have shrunk so much that  period-4 cells are now the largest.

\begin{figure}
  \includegraphics[width = 0.49\textwidth]{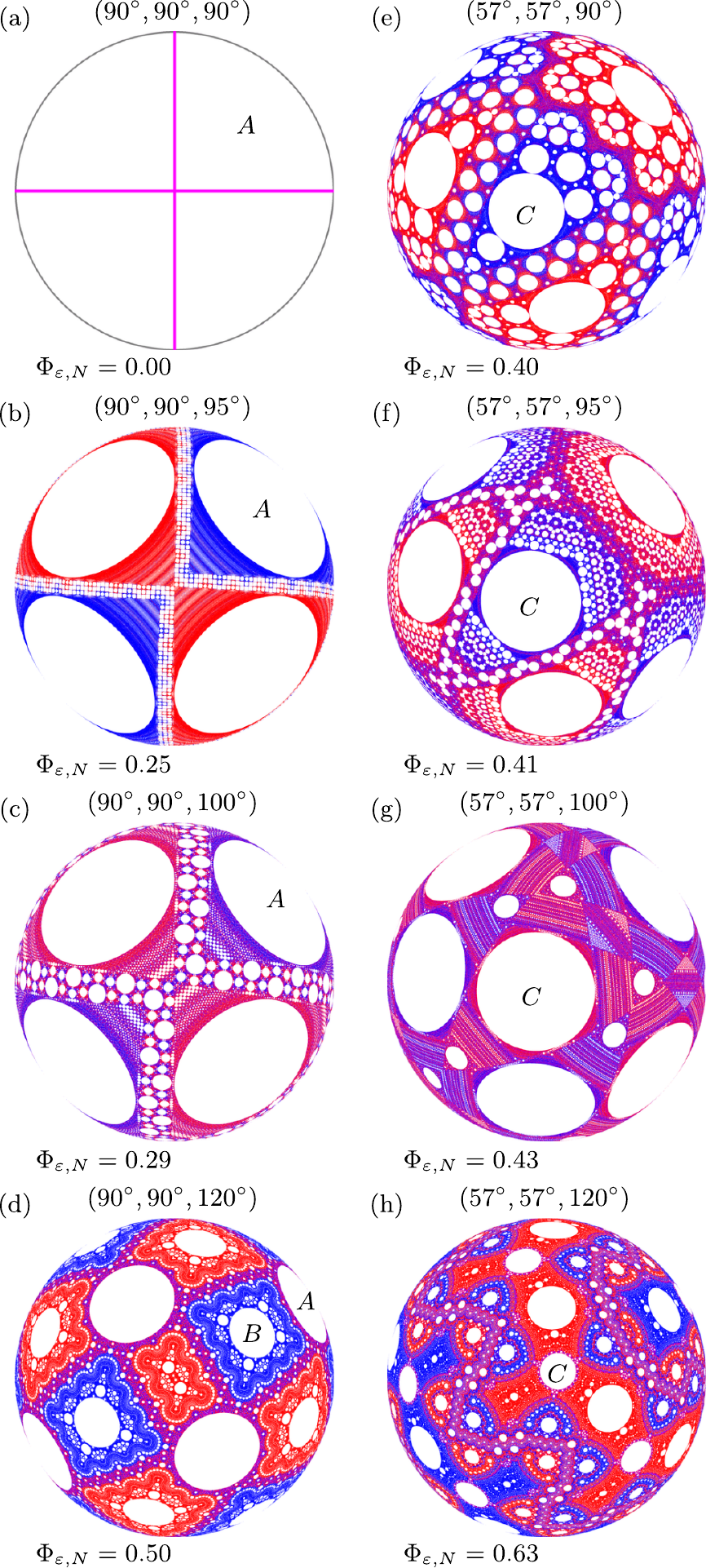}
  \caption{The exceptional set \(\tilde{E}_{\varepsilon,N}\) for increasing \(\gamma\) viewed from the negative \(y\)-axis with the \(z\)-axis pointing toward the top of the page [\(\varepsilon = 1\times 10^{-4}\) and \(N = 2\times 10^4\) except for \((90^{\circ}, 90^{\circ}, 90^{\circ})\) which uses \(\varepsilon = 0.01\) to yield visible line thickness]. (a-d) \((90^{\circ},90^{\circ},\gamma)\) and (e-h) \((57^{\circ},57^{\circ},\gamma)\), \(\gamma\) increases downward. One of the four period-2 cells for the \((90^{\circ},90^{\circ},\gamma)\) protocols is labeled \(A\). A period-4 cell is labeled \(B\) in (d) and a period-3 cell is labeled \(C\) in (e-h).}
  \label{fig:change-gamma}
\end{figure}

The \((57^{\circ},57^{\circ},100^{\circ})\) protocol has fewer noticeably large non-mixing islands than the other protocols shown in Fig.~\ref{fig:change-gamma}(e-h). This is due to the protocol's proximity in the protocol space to a resonant, periodically non-mixing, protocol at \((57^{\circ},57^{\circ},99.275...^{\circ})\)
 shown in Fig.~\ref{fig:57resonance}. This resonant PWI is complete after only 160 iterations \footnote{This also implies that \(E = \bar{E}\) and the exceptional set has zero measure.}. The smaller, circle-like cells are actually period-5 32 sided polygons; a part of a polygonal cell is shown in the detail on the right of Fig.~\ref{fig:57resonance}. The polygonal cells tile the HS and result in zero coverage, which reveals that \(\Phi_{\varepsilon,N}\) does not monotonically increase with \(\gamma\). The near miss of the \((57^{\circ},57^{\circ},100^{\circ})\) protocol to a resonant protocol explains its densely packed cutting lines. Cutting lines gradually fill in closely to one another as \(\tilde{E}\) evolves, but the cutting lines slightly miss returning upon themselves by small amounts due to nearly rational internal rotation within cells. Of course, even though the densely packed lines suggest a high degree of mixing in those regions of the HS, it would take many iterations for this to occur. This behavior will be exploited later to locate resonant protocols by their nearly resonant neighbors that have very low, near-resonant coverage for small \(N\).

\begin{figure}
  \includegraphics[width = 0.405\textwidth]{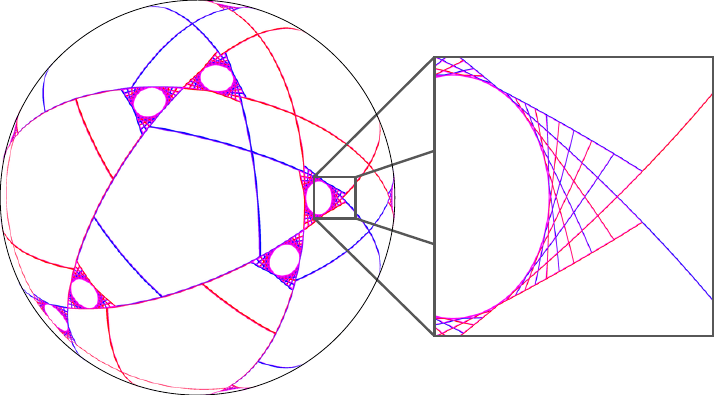}
  \caption{A resonance, or locally minimally mixing protocol, occurs for the \((57^{\circ}, 57^{\circ}, 99.3^{\circ})\) protocol. \(E\) is complete after only 160 iterations.  \(E\) forms a polygonal tiling, with large regular pentagons, small regular 32-gons, and irregular polygons (mostly or all quadrilaterals). Compare this simple cutting line structure to that of the densely packed \((57^{\circ},57^{\circ},100^{\circ})\) protocol in Fig.~\ref{fig:change-gamma}(g) where \(\gamma\) differs by only \(0.7^{\circ}\).
  }
  \label{fig:57resonance}
\end{figure}

\subsection{Special symmetries with orthogonal axes and breaking them}

As \(\gamma\) is increased from \(\gamma = 90^\circ\) in Fig.~\ref{fig:change-gamma}(a) to (b), it is clear that a symmetry is broken. The orthogonal axes case has additional barriers to mixing not present in the non-orthogonal case due to additional symmetries that occur when \(\gamma = 90^{\circ}\), specifically symmetries across \(\beta = 90^\circ\) (Eq.~\ref{eq:symm-beta-gamma}) and across \(\alpha = 90^\circ\) (Eq.~\ref{eq:symm-alpha-gamma}). An example of this appears for \((90^{\circ}, \beta, 90^{\circ})\) in the orthogonal axes case shown in Fig.~\ref{fig:tax_orthog} where there is a clear barrier to mixing regardless of the value of \(\beta\). The  left and right sides for the colored initial condition do not mix, even though \(\tilde{E}\) on the right appears to indicate mixing for \(\beta = 60^\circ\) and \(\beta = 30^\circ\). Likewise, for \((\alpha, 90^{\circ}, 90^{\circ})\) there would be a horizontal barrier to mixing (not shown, but analogous to the \((90^\circ, \beta, 90^\circ)\) case shown in Fig.~\ref{fig:tax_orthog} except with a horizontal instead of vertical barrier). The \((90^{\circ}, 90^{\circ}, 90^{\circ})\) protocol in Fig.~\ref{fig:change-gamma}(a) has both horizontal and vertical barriers, resulting in a non-mixing system. The barriers to mixing result directly from the inability of the two rotational axes to interact with one another due to the symmetry created by orthogonal axes. The barriers to mixing are not evident in the structure of \(\tilde{E}\). For example, in the right column of Fig.~\ref{fig:tax_orthog} even though there is a symmetry in \(\tilde{E}\) about the vertical \(z\)-axis, there is no evidence that there is no mixing between the left and right halves of the hemisphere.

\begin{figure}
  \includegraphics[width = 0.5\textwidth]{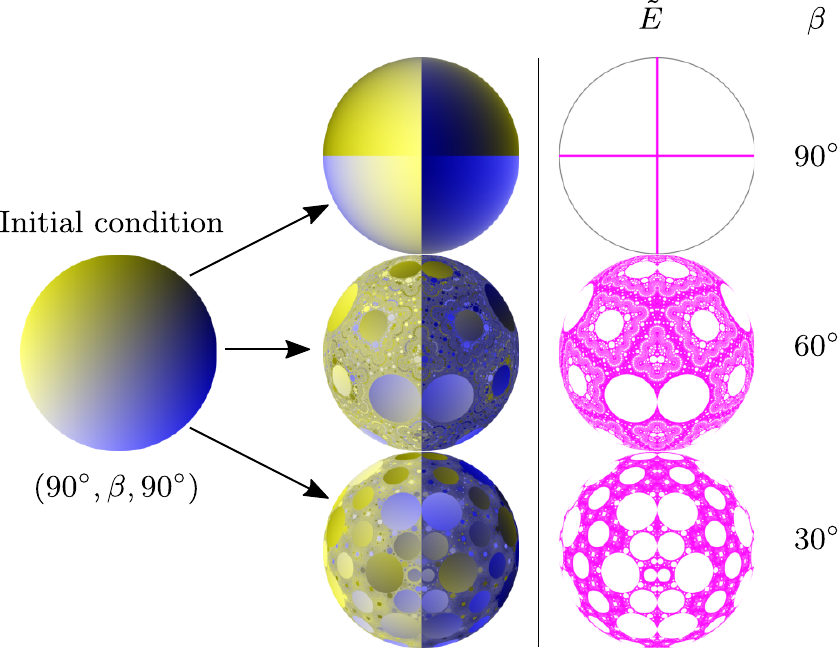}
  \caption{Barriers to mixing occur in the \((90^{\circ}, \beta, 90^\circ)\) case shown after \(2\times 10^4\) iterations of the initial condition for \(\beta = 30^{\circ}, 60^{\circ}, 90^{\circ}\). Although \(\tilde{E}\), shown on the right, appears to indicate mixing throughout the domain (except, of course, the cells) based on the fractional coverage of \(\tilde{E}\), the symmetry of the system resulting from orthogonal rotation axes  creates a barrier to mixing between the left and right halves of the hemisphere. 
  }
  \label{fig:tax_orthog}
\end{figure}

Figure~\ref{fig:mixing_init} demonstrates that this left-right barrier to mixing is not apparent for all initial conditions. Specifically, when the initial condition itself has left-right symmetry, as demonstrated in Fig.~\ref{fig:mixing_init}(a), the barrier to transport is not evident. For any other initial condition, such as those in Figs.~\ref{fig:mixing_init}(b) and (c), the barrier is evident. The barrier is most notable when there is large variation in the initial condition in the direction orthogonal to the barrier, as shown in Fig.~\ref{fig:mixing_init}(c).


\begin{figure}
  \includegraphics[width = 0.45\textwidth]{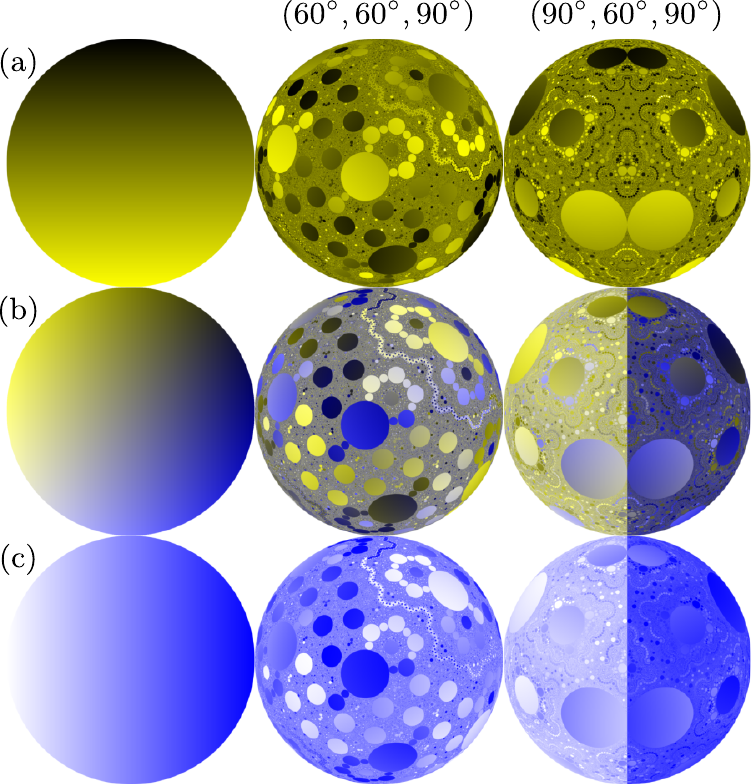}
  \caption{Various initial conditions (left column) under \((60^{\circ},60^{\circ},90^{\circ})\) and \((90^{\circ},60^{\circ},90^{\circ})\) after \(2 \times 10^4\) iterations show sensitivity of mixing to initial conditions. Initial condition varied in (a) vertical \(z\)-direction, (b) both \(x\)- (blue) and \(z\)- (yellow) directions, and (c) horizontal \(x\)-direction.}
  \label{fig:mixing_init}
\end{figure}

For any fixed \(\gamma \neq 90^\circ\), the symmetries evident in Fig.~\ref{fig:tax_orthog} (and shown in mathematical detail in Appendix \ref{sec:symmetries},  Eqs.~\ref{eq:symm-beta-gamma} and \ref{eq:symm-alpha-gamma}) are broken and no longer create degenerate mixing cases like those shown for \(\gamma = 90^\circ\), since \(\gamma \neq -\gamma \mod 180^\circ\) essentially removes a reflection symmetry from the system. Although the reflection symmetries across \((90^\circ,\beta, 90^\circ)\) and \((\alpha, 90^\circ, 90^\circ)\) are the only symmetries broken by non-orthogonal axes, breaking them opens up a much wider range of unique protocols to be considered. For example, the completely non-mixing exceptional set shown in Fig.~\ref{fig:change-gamma}(a), which is the result of both of these symmetries, becomes a mixing protocol by eliminating the aforementioned symmetries that occur when \(\gamma = 90^\circ\).


\section{Resonances in the fractional coverage \(\Phi_{\varepsilon,N}\)} \label{sec:phase_space}

The fat-line method for constructing \(\tilde{E}\) described in Section~\ref{sec:tildeE} allows us to  efficiently generate \(\Phi_{\varepsilon,N}\) to identify resonances in the fractional coverage over a broad range of protocols. To this end, the number of iterations used to identify resonances is intentionally kept low, \(N= 500\), to better reveal resonant protocols through their nearly resonant neighbors in the protocol space as explained shortly.

\subsection{Revisiting orthogonal axes, \(\gamma = 90^{\circ}\)}

Although fractional coverage of \(\bar{E}\) for the orthogonal axis case (\(\gamma = 90^{\circ}\)) has been previously investigated \cite{Park2017, Park2016, Smith2017}, it is informative to contrast some of its features with those of \(\bar{E}\) for non-orthogonal axes. Since periodic rotations about the \(z\)-axis (by \(\tilde{M}_\theta^a\)) have period \(180^\circ\), \(\Phi_{\varepsilon,N}\) is examined for \(\gamma = 90^{\circ}\) over the range \(0 \le \alpha, \beta \le 180^{\circ}\). Figure~\ref{fig:phasespaces}(a) shows \(\Phi_{0.01,500}\) where dark gray represents low coverage and light gray represents high coverage. A value of \(\Phi_{0.01,500} = 1\) (white) means that \(\tilde{E}_{0.01,500}\) for that protocol completely covers the HS with fattened cutting lines by 500 iterations, while \(\Phi_{0.01,500} \approx 0\) (black) indicates no coverage. Of course, due to the fattening of \(\bar{E}\) into \(\tilde{E}_{0.01}\), values of \(\Phi_{0.01,500}\) can never be exactly zero, and without a way to determine a sufficient \(N\) \textit{a priori} such that \(\tilde{E}\) covers \(\bar{E}\), there is no way to know if \(\tilde{E}_{0.01,500}\) is complete (satisfies  Eq.~\ref{eq:fat_coverage}) or if \(\Phi_{0.01,500}\) is an over- or under-estimate. However, when comparing \(\Phi_{0.01,500}\) with a more accurate estimate, \(\Phi_{1 \times 10^{-3},2\times 10^4}\), for more than 140,000 protocols with \(\gamma = 90^\circ\), \(\Phi_{0.01,500}\) typically overestimates \(\Phi_{1 \times 10^{-3},2\times 10^4}\) (for 84.0\% of protocols), but with only a very small number (3.5\%) differing by more than 5\% in value. On average, \(\Phi_{0.01,500}\) overestimates \(\Phi_{1 \times 10^{-3},2\times 10^4}\) by 1.5\%. Protocols for which \(\Phi_{0.01,500}\) underestimates \(\Phi_{1 \times 10^{-3},2\times 10^4}\) by more than 5\% account for only 0.6\% of all the protocols.

\begin{figure*}
  \includegraphics[width = 0.73\textwidth]{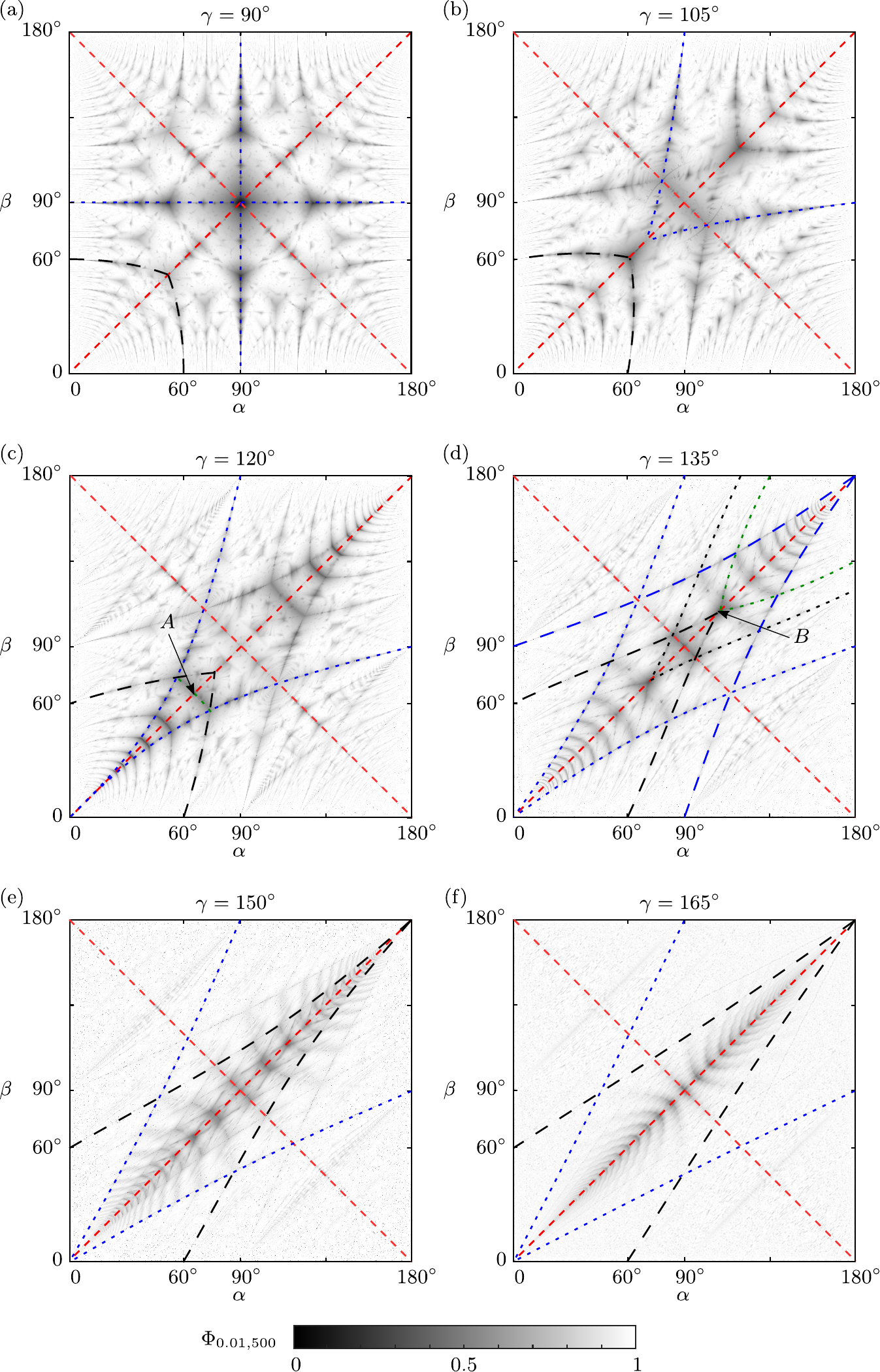}
  \caption{Fractional coverage, \(\Phi_{0.01,500}\), vs. rotation angles for \(\gamma = 90^{\circ}\) to \(165^{\circ}\) in \(15^{\circ}\) increments after 500 iterations. Increments in \(\alpha\) and \(\beta\) are \(0.33^\circ\). (a) For \(\gamma = 90^\circ\), reflection symmetry lines are shown as red (\(\alpha + \beta = 180^\circ\) and \(\alpha = \beta\)) and blue (\(\alpha = 90^\circ\) and \(\beta = 90^\circ\)) dashed lines. Symmetry along the red dashed lines is maintained for non-orthogonal axes while the blue dashed symmetries are lost. (a-f) Black dashed curves mark the locations of Arnold tongues originating at \((60^\circ, 0, \gamma)\) and \(( 0,60^\circ, \gamma)\). Blue dashed curves mark Arnold tongues starting at \(( 90^\circ,180^\circ, \gamma)\) and \(( 180^\circ,90^\circ, \gamma)\). (c) The protocol \((65.5^\circ, 65.5^\circ, 120^\circ)\) is labeled \(A\) and lies on a dark band between Arnold tongue intersections. (d) The protocol labeled \(B\) lies close to the intersection of the Arnold tongues from \((60^\circ, 0, \gamma)\), \(( 0,60^\circ, \gamma)\) and \((135^\circ, 180^\circ, \gamma)\) , \((180^\circ, 135^\circ,  \gamma)\) (green dashed curves). The short-dashed lines show the corresponding mirror Arnold tongues for long-dashed lines of the same color. The intersection of short-dashed and long-dashed lines are often locally resonant.
  }
  \label{fig:phasespaces}
\end{figure*}

Figure~\ref{fig:phasespaces}(a) has many local minima, which appear as dark regions, and represent protocols at or near resonances \cite{Smith2017}. Resonances are located along Arnold tongues that originate from rational fractions of \(\pi\) along the axes. Two such Arnold tongues are highlighted with black dashed curves emanating from \((60^\circ, 0, 90^\circ)\) and \((0, 60^\circ, 90^\circ)\). The Arnold tongues are centered along curves where periodic points of the PWI are equally spaced between two of the three cutting lines that form the atoms of the PWI  \cite{Smith2017}. Resonances occur whenever periodic points of the PWI, each of which corresponds with a particular periodic itinerary, are maximally far from the cutting lines defining the PWI and result in large cells in \(E\). In this way, resonances occur at the intersection of the three curves that represent a periodic point equidistant between two cutting line borders of an atom when \(\gamma = 90^\circ\) \cite{Smith2017}. For \(\gamma = 90^\circ\), the blue dashed lines, corresponding to Arnold tongues from \((90^\circ, 180^\circ, \gamma)\) and \((180^\circ, 90^\circ, \gamma)\) in Fig.~\ref{fig:phasespaces}, are also lines of symmetry in the protocol space.

Smith et al.\ \cite{Smith2017} have explained much of the structure in the \(\gamma = 90^\circ\) protocol space. Protocols further away from a resonant protocol in the \((\alpha, \beta, 90^{\circ})\) protocol space in Fig.~\ref{fig:phasespaces}(a) have periodic points closer to cutting lines. As a periodic point approaches a cutting line, its associated cell shrinks until the periodic point encounters a cutting line and the cell is annihilated (i.e., the radius goes to zero and the point is no longer a periodic point). Thus, not only is there a protocol that maximizes the size of the cell associated with each periodic itinerary, but there is also a well defined region within the protocol space where the cell exists (i.e., for protocols outside of this region the periodic itinerary between atoms defining the periodic cell is not possible). An example cell with a 411 itinerary (i.e.\ the cell travels periodically through the atoms \(P_4 \to P_1 \to P_1 \to P_4 \cdots\)) is labeled \(C\) in Fig.~\ref{fig:change-gamma}(e) and the size of the cell is maximized at the resonance shown in Fig.~\ref{fig:57resonance} when the angle between the axes increases to \(99.3^\circ\). The cell shrinks as the protocol moves away from this resonance.

Resonances along \(\alpha = \beta\) and \(\alpha + \beta = 180^\circ\), corresponding to the red dashed diagonal lines in Fig.~\ref{fig:phasespaces}(a), at the intersection of Arnold tongues have exceptional sets that are polygonal tilings, similar to \(\bar{E}\) shown in Figs.~\ref{fig:change-gamma}(a) and \ref{fig:57resonance}, and, thus,  have zero coverage \cite{Smith2017}. The protocols nearby in protocol space have nearly resonant structure such that \(\Phi_{\varepsilon,N}\) grows slowly with \(N\). By using \(N=500\), these structures do not complete and thus resemble the resonant protocols. When the protocol space is sampled, hitting a resonant structure exactly is generally impossible \footnote{As will be shown later, the resonances generally occur for protocols whose angles are incommensurate with \(\pi\). Therefore, it is never possible to find the resonances exactly if the protocol space is sampled on a uniform grid.}, but using a low number of iterations effectively fattens resonant regions of the protocol space allowing them to be observed more easily. Changing the number of iterations does not change the locations of resonances in the protocol space.

\subsection{Non-orthogonal axes} \label{sec:non-orthog}

The fractional coverage of \(\tilde{E}\) for several values of \(\gamma > 90^{\circ}\) is shown in Fig.~\ref{fig:phasespaces}(b-f) \footnote{See the Supplementary Material at  [URL will be inserted by publisher] for a video sequence of coverage maps as \(\gamma\) is varied from \(0.5^\circ\) to \(179.5^\circ\) in \(0.5^\circ\) increments.}. A previous study \cite{Juarez2012}, using PWI and a mixing metric based on the center of mass of seeded particles in the HS, demonstrated similar patterns for the quadrant \(0 \le \alpha, \beta \le 90^{\circ}\). Patterns outside of \(0 \le \alpha, \beta \le 90^{\circ}\) have not been previously examined. Due to symmetries in the system, the patterns for angles between rotation axes of \(180^\circ - \gamma\) are equivalent to those in Fig.~\ref{fig:phasespaces} when reflected vertically or horizontally about \(\alpha = 90^\circ\) or \(\beta = 90^\circ\), respectively. The protocol spaces for \(\gamma = 0\) or \( 180^{\circ}\) are not shown since they are degenerate cases of rotation about a single axis and display no interesting features.

Obvious features of the coverage, \(\Phi_{0.01,500}\), in protocol space for orthogonal axes in Fig.~\ref{fig:phasespaces}(a) are the lines of symmetry, shown as red and blue dashed lines. The red dashed lines represent symmetry between \((\alpha, \beta, \gamma)\) and \((\beta,\alpha,\gamma)\) (Eq.~\ref{eq:symm-reverse-time}) and between \((\alpha, \beta, \gamma)\) and \((\pi - \alpha, \pi-\beta, \gamma)\) (Eq.~\ref{eq:symm-negate-two}). For non-orthogonal axes (\(\gamma\neq 90^\circ\)) these symmetries extend as reflection planes along \(\alpha + \beta = 180^\circ\) and \(\alpha = \beta\). The blue dashed lines in Fig.~\ref{fig:phasespaces}(a) represent a symmetry  between \((\abg)\) and \((\pi - \alpha,\beta,\pi - \gamma)\) (Eq.~\ref{eq:symm-alpha-gamma}) and between  \((\alpha, \beta, \gamma)\) and \((\alpha,\pi - \beta,\pi -\gamma)\) (Eqs.~\ref{eq:symm-beta-gamma} and \ref{eq:symm-negate-gamma}). When \(\gamma = 90^{\circ}\), these symmetries are reflection symmetries across the blue lines, since \(\gamma = -\gamma \mod 180^\circ\) when \(\gamma = 90^\circ\). However, this reflection does not extend to protocols with non-orthogonal axes, which is evident in Figs.~\ref{fig:phasespaces}(b-f). A full exposition on the symmetries of the PWI can be found in Appendix~\ref{sec:symmetries}.

The  Arnold tongues along which resonances are located and which originate from rational fractions of \(\pi\) along \(\alpha = 0\) and \(\beta = 0\) for orthogonal rotation axes in Fig.~\ref{fig:phasespaces}(a) are maintained for non-orthogonal axes, but as \(\gamma\) increases, these tongues ``tilt'' (to the right for those originating from \(\beta = 0\) and upward for those originating from \(\alpha = 0\)), eventually approaching straight lines as \(\gamma \to 180^\circ\). An example pair of tongues starting at \((60^\circ,0, \gamma)\) and \((0,60^\circ,\gamma)\) is highlighted with dashed black curves in Fig.~\ref{fig:phasespaces}(a-f). Another pair originating from \((90^\circ, 180^\circ,\gamma)\) and \((180^\circ,90^\circ,\gamma)\) is marked with dashed blue curves. When \(\gamma\) is decreased from \(90^{\circ}\) (not shown), the Arnold tongues tilt in the opposite direction and similarly approach straight lines.

As \(\gamma\) increases, the tilting of the Arnold tongues means that the resonant structures (dark regions) that lie within the \(\alpha = \beta\) and \(\alpha + \beta = 180^{\circ}\) symmetry planes move along the symmetry planes, following the intersections of Arnold tongues as they tilt. An example is the resonance associated with the polygonal tiling for the \((57^\circ, 57^\circ, 99.3^\circ)\) protocol in Fig.~\ref{fig:57resonance}.  The Arnold tongues starting at \((60^\circ, 0, 90^\circ)\) and \((0, 60^\circ, 90^\circ)\) in Fig.~\ref{fig:phasespaces}(a) tilt as \(\gamma\) increases, so that their intersection (the cusp at the intersection of the black dashed curves) moves to the right along the \(\alpha=\beta\) symmetry line.  When this intersection passes through the \((57^\circ, 57^\circ, 99.3^\circ)\) protocol between Fig.~\ref{fig:phasespaces}(a) and \ref{fig:phasespaces}(b), the resonance occurs.

In addition, as the tongues tilt, new intersections with other tongues originating from \(\alpha = 180^\circ\) or \(\beta = 180^\circ\) are created, resulting in additional resonant structures. These additional resonances spread out symmetrically from the \(\alpha = \beta \) symmetry plane as \(\gamma\) is increased. 
This can be seen between \(\gamma = 105^\circ\) and \(\gamma = 120^\circ\) in Figs.~\ref{fig:phasespaces}(b,c). At exactly \(\gamma = 108^\circ\) (not shown), the black and blue dashed lines intersect at their cusp and increasing \(\gamma\) further introduces two intersection points on either side of \(\alpha = \beta\). When \(\gamma = 120^\circ\) a dark band of low coverage protocols between these intersections, highlighted with a green dashed line in Fig.~\ref{fig:phasespaces}(c), is evident. The protocol \(A\), \((65.5^\circ, 65.5^\circ, 120^\circ)\), lies along this dark band between Arnold tongue intersections (blue and black intersecting), and will be discussed in more detail later. This band is defined by the pair of Arnold tongues starting at \((60^\circ,0^\circ,120^\circ)\) and \((0^\circ,60^\circ,120^\circ)\) labeled with black dashed curves, and the pair of Arnold tongues starting at \((90^\circ,180^\circ,120^\circ)\) and \((180^\circ,90^\circ,120^\circ)\), labeled with blue dashed curves. More of these dark bands appear between the intersections of two pairs of Arnold tongues as \(\gamma\) is varied, and many more are evident when \(\gamma = 120^\circ\) along the length of the \(\alpha = \beta\) symmetry plane.

Tongues from \((0,90^\circ, \gamma)\) and \((90^\circ, 0, \gamma)\) [highlighted with long-dashed blue lines in Fig.~\ref{fig:phasespaces}(d)] intersect their symmetric counterparts [highlighted  with short-dashed blue lines in Figs.~\ref{fig:phasespaces}(b-f)] from \((180^\circ,90^\circ, \gamma)\) and \((90^\circ,180^\circ, \gamma)\) as \(\gamma\) is increased from \(90^\circ\) creating new resonances along \(\alpha + \beta = 180^\circ\) at their intersection. Likewise, the tongues originating from \((60^\circ, 0, \gamma)\) and \((0, 60^\circ, \gamma)\) labeled with dashed black curves also create new resonances along \(\alpha + \beta = 180^\circ\) when they intersect their mirror images [highlighted with short-dashed black lines Fig.~\ref{fig:phasespaces}(d)] from \((120^\circ, 180^\circ, \gamma)\) and \((180^\circ, 120^\circ, \gamma)\) as \(\gamma\) increases from \(\gamma = 120^{\circ}\) [Fig.~\ref{fig:phasespaces}(c)] to \(\gamma = 135^{\circ}\) [Fig.~\ref{fig:phasespaces}(d)]. The rate at which new intersections occur (and create dark band between them) as \(\gamma\) increases grows such that by \(\gamma = 165^{\circ}\) many of these spreading resonances and their connecting dark bands have blended together to form nested curves of low coverage in the protocol space in Fig.~\ref{fig:phasespaces}(f).

\subsection{ Polygonal tilings in the \(\alpha = \beta\) symmetry plane}

The intersections of symmetrical pairs of Arnold tongues, e.g.\ the two sets of black dashed curves in Fig.~\ref{fig:phasespaces}(d), occur exclusively on the diagonal symmetry planes \(\alpha + \beta = 180^\circ\) and \(\alpha = \beta\). The fractional coverage of \(\tilde{E}\), \(\Phi_{0.01,500}\) along the \(\alpha = \beta\) diagonal symmetry plane is shown in Fig.~\ref{fig:full-flame} for varying \(\gamma\). For context, the dark region at the center of Fig.~\ref{fig:phasespaces}(a), \((90^\circ,90^\circ,90^\circ)\), matches the dark region at the center of Fig.~\ref{fig:full-flame}(a), also \((90^\circ,90^\circ,90^\circ)\).


\begin{figure*}
  \includegraphics[width = 0.9\textwidth]{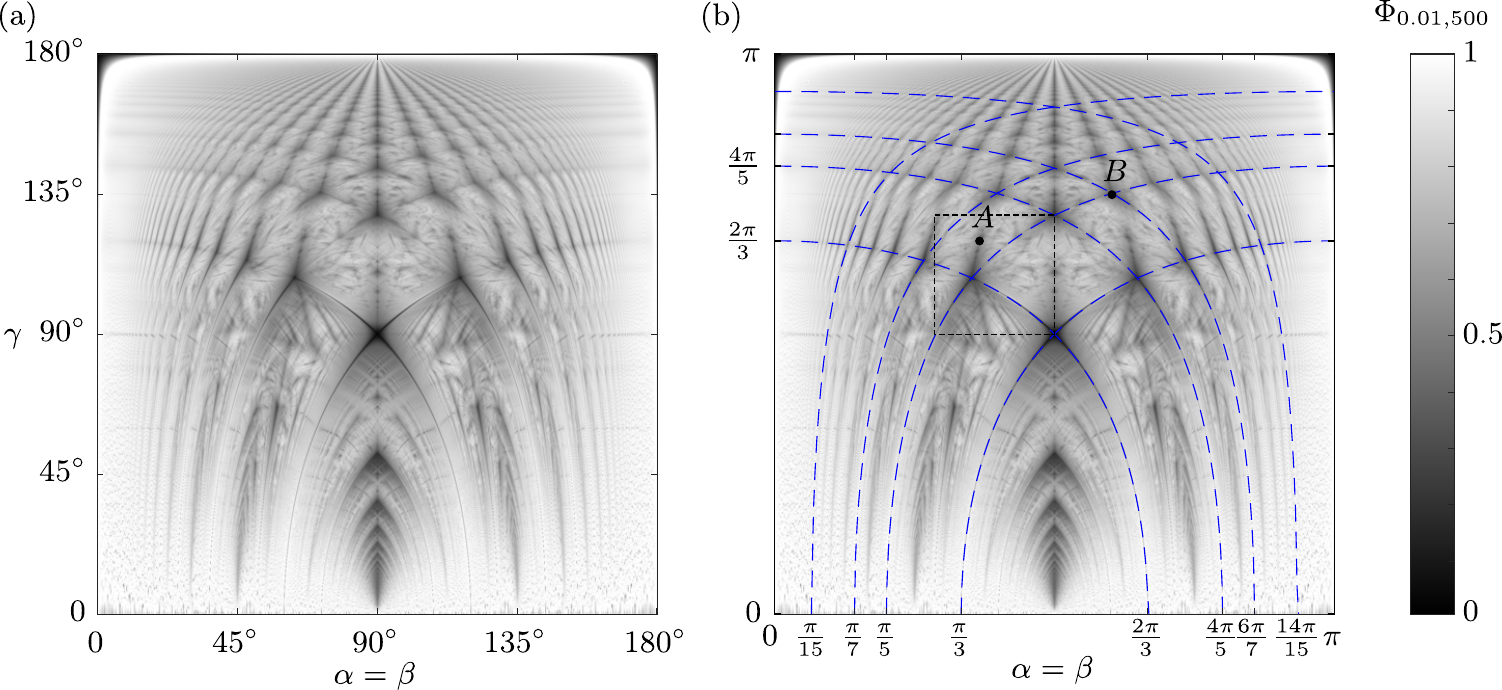}
  \caption{(a) Fractional coverage, \(\Phi_{0.01,500}\), for protocols along the \(\alpha=\beta\) symmetry plane and varying axes angle, \(\gamma\). The upper half, \(\gamma \ge 90^{\circ}\), corresponds with behavior along the \(\alpha = \beta\) line in Fig.~\ref{fig:phasespaces} and the lower half, \(\gamma \le 90^\circ\), corresponds with the \(\alpha + \beta = 180^{\circ}\) line in Fig.~\ref{fig:phasespaces} due to the symmetries in Eqs.~\ref{eq:symm-beta-gamma} and \ref{eq:symm-alpha-gamma}. 
  (b) The same symmetry plane with resonant curves defined by Eqs.~\ref{eq:branches-r} and \ref{eq:branches-l} for \(j = 1,2,3,7\). The boxed region is shown in Fig.~\ref{fig:branchsamples}. The protocol labeled \(A\), \((65.5^\circ, 65.5^\circ, 120^\circ)\), corresponds with Fig.~\ref{fig:phasespaces}(c). The protocol labeled \(B\) near the intersection of the branches from \((\pi/5,\pi/5,0)\) and \((6\pi/7,6\pi/7,0)\) corresponds with \(B\) in Fig.~\ref{fig:phasespaces}(d).
  }
  \label{fig:full-flame}
\end{figure*}

The most obvious features of the \(\alpha = \beta\) symmetry plane shown in Fig.~\ref{fig:full-flame} are the dark branches \footnote{To differentiate between the Arnold tongues shown in Fig.~\ref{fig:phasespaces} for fixed \(\gamma\) and the tongues in the \(\alpha = \beta\) plane, we refer to tongues in the \(\alpha = \beta\) plane as \textit{branches}.} of resonant structures sprouting from odd divisors of \(\pi\) [i.e.\ \(\pi/(2j +1)\) where \(j > 0\) is an integer] at \(\gamma = 0\) and their complements coming from \(\pi-[\pi/(2j + 1)]\), both of which are highlighted with dashed blue lines in Fig.~\ref{fig:full-flame}(b). These branches follow the intersections of pairs of Arnold tongues; for example, the branch originating from \((2\pi/3, 2\pi/3, 0)\) in Fig.~\ref{fig:full-flame}(b) follows the cusp intersection of the Arnold tongues labeled with blue short-dashed lines in Fig.~\ref{fig:phasespaces}. As mentioned earlier, for orthogonal rotation axes, these intersections mark the locations of polygonal exceptional set structures \cite{Smith2017}.

The curves defining these resonant branches in the symmetry plane originating from \(\pi - [\pi/(2j+1)]\) and highlighted with blue dashed curves in Fig.~\ref{fig:full-flame}(b) are level curves of \(\cos(\alpha/2)\cos(\gamma/2)\) (verification in Appendix \ref{sec:non-ortho_angles}), specifically
\begin{equation}
  \cos\left(\frac{\alpha}{2}\right) \cos\left(\frac{\gamma}{2}\right) = \sin\left[\frac{\pi}{2 (2j+1)}\right]. \label{eq:branches-r}
\end{equation}
The curves defining the branches originating from \(\pi / (2j+1)\) are the mirror images across the \(\alpha = \beta = \pi/2\) line,
\begin{equation}
  \sin\left(\frac{\alpha}{2}\right) \cos\left(\frac{\gamma}{2}\right) = \sin\left[\frac{\pi}{2 (2j+1)}\right]. \label{eq:branches-l}
\end{equation}
For example, the branch originating from \((2\pi/3, 2\pi/3, 0)\) in Fig.~\ref{fig:full-flame}(b) [\((2j+1) = 3\) in Eq.~\ref{eq:branches-r}] lies along the curve \(\sin\left(\frac{\alpha}{2}\right) \cos\left(\frac{\gamma}{2}\right) = \sin\left(\frac{\pi}{6}\right)\) and follows the intersection of the blue dashed curves in Fig.~\ref{fig:phasespaces} within the \(\alpha = \beta\) plane. Likewise, the intersection of the highlighted Arnold tongues starting at \((60^\circ, 0, 0)\) and \((0, 60^\circ, 0)\) in Fig.~\ref{fig:phasespaces} (the intersection of black dashed curves) follows the curve originating at \((\pi/5, \pi/5, 0)\), such that \(2j+1 = 5\), in Fig.~\ref{fig:full-flame}(b). These curves appear to be integral to most of the structure in the symmetry plane.

One of the features of the polygonal tilings found previously along \(\alpha = \beta\) for orthogonal axes, \(\gamma = 90^\circ\), is the presence of (\(2j+1\))-gons \cite{Smith2017}. These polygons are cells that rotate internally by \(2j\pi/(2j+1)\) each time they complete their period-\((j+1)\) itinerary through the HS. The branches defined by Eqs.~\ref{eq:branches-r} and \ref{eq:branches-l} shown in Fig.~\ref{fig:full-flame}(b) contain polygonal resonant structures from \(\gamma = 0\) to their first intersection with another branch which is always the \(\pi/3\) or \(2\pi/3\) branch corresponding to \(2j+1 = 3\). To better illustrate this, Fig.~\ref{fig:branchsamples} includes some example exceptional sets along two of the prominent branches from the boxed region in Fig.~\ref{fig:full-flame}(b) focusing on the \(2j+1 = 5\) branch originating from the left and the \(2j+1 = 3\) branch originating from the right. Polygonal tilings occur for all \(\gamma \le 108^\circ\) along these two branches. A detailed analysis showing that rational internal rotation of cells occurs along these curves to create \((2j+1)\)-gons is presented in Appendix~\ref{sec:non-ortho_angles}.

These two branches intersect at \(\gamma = 108^\circ\). Along the \(2j+1 = 5\) branch originating from the left in Fig.~\ref{fig:branchsamples}, period-3 pentagons are present in the exceptional set. Along the \(2j+1 = 3\) branch originating from the right in Fig.~\ref{fig:branchsamples}, period-2 triangles are present in the exceptional set. The intersection of these two branches at \((63.435...^\circ, 63.435...^\circ, 108^\circ)\) has an exceptional set that is a polygonal tiling composed entirely of period-3 pentagons and period-5 triangles \footnote{There is a point along the \(j = 1\) branches at which triangular cells are annihilated and return with a different periodicity.}, i.e.\ it is half of an icosidodecahedron. The PWI for this protocol reconstructs the initial HS after only 15 iterations which is the least common multiple of 5 and 3, the periods of polygonal cells. The shortest-period cells present in the exceptional set for protocols near these branch intersections have periods equal to the sum of the integer labels for the branches (\(j_L\) from the left in Eq.~\ref{eq:branches-l} and \(j_R\) from the right in Eq.~\ref{eq:branches-r}). Since \((63.4^\circ, 63.4^\circ, 108^\circ)\) occurs at the intersection of the \(j_L = 2\) (left) and \(j_R = 1\) (right) branches, the shortest period cells, the pentagons, are period \(j_L + j_R = 3\). Similarly, a tiling composed of only period-2 triangles occurs at the intersections of \(j_L = 1\) and \(j_R = 1\), \((90^\circ, 90^\circ, 90^\circ)\) in Fig.~\ref{fig:branchsamples}, with nearby protocols having large period-2 cells. Note that either \(j_L = 1\) or \(j_R = 1\) for any polygonal tiling since polygonal tilings do not exist along these branches past the \(j = 1\) intersections, but this is valid for other protocols. For example, the protocol labeled \(B\) in Figs.~\ref{fig:phasespaces}(d) and \ref{fig:full-flame}(d), at the intersection of Arnold tongues labeled with green and black lines, has circular period-5 cells for its shortest period cells due to its proximity to the intersection of \(j_L = 2\) and \(j_R = 3\) branches. The \((63.4^\circ, 63.4^\circ, 108^\circ)\), the \((90^{\circ},90^{\circ},90^{\circ})\) protocol, and the \((70.5^\circ,70.5^\circ,60^\circ)\) protocol are the only protocols that are related to regular solids representing the projection onto the HS of an icosidodecahedron, an octahedron, and a cuboctahedron, respectively. It is likely that these are the only such solids that can be created using this form of PWI, since they are the only solids with projections onto the HS that are constructed exclusively with great-circles.

\begin{figure*}
  \includegraphics[width = 0.9\textwidth]{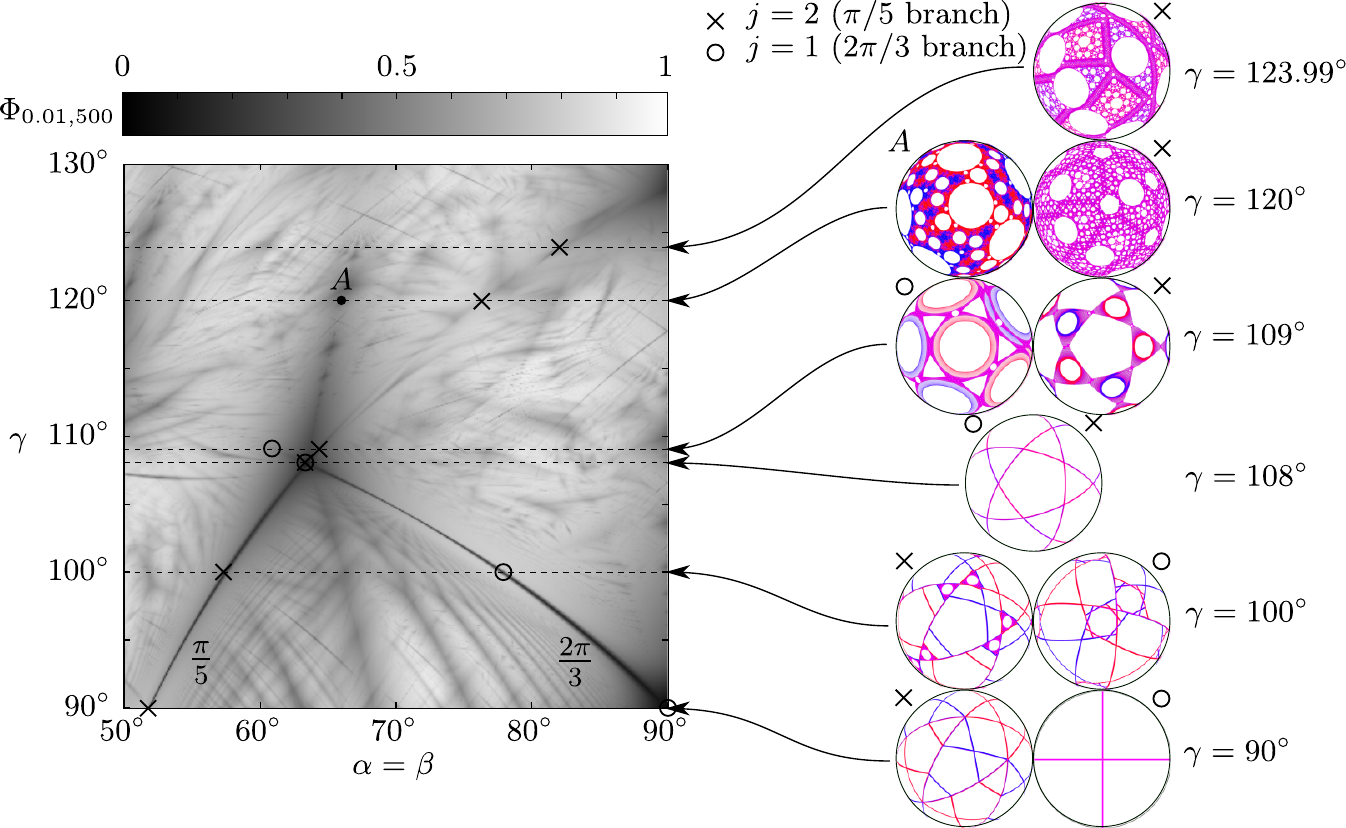}
  \caption{A detailed view of the region where the \(\frac{\pi}{5}\) and \(\frac{2\pi}{3}\) branches intersect marked by the dashed black box in Fig.~\ref{fig:full-flame}. Some of the \(\tilde{E}\) structures along these branches after \(N=2\times 10^4\) iterations are shown on the right with their locations in the protocol space marked by \(\times\) for the \(\pi/5\) branch (\(2j+1 = 5\)) and \(\circ\) for the \(2\pi/3\) branch (\(2j+1 = 3\)). The protocol labeled \(A\), \((65.5^\circ, 65.5^\circ, 120^\circ)\), corresponds with Figs.~\ref{fig:phasespaces}(c) and \ref{fig:full-flame}(b). \(\varepsilon = 0.01\) for \(\gamma \le 108^\circ\) and \(\varepsilon = 5\times 10^{-4}\) for \(\gamma \ge 109^\circ\).
  }
  \label{fig:branchsamples}
\end{figure*}


The branches originating from \(\pi/3\) and \(2\pi/3\) in Fig.~\ref{fig:full-flame} are unique in that they intersect each other, their own complements, before other branches as \(\gamma\) increases. Furthermore, they continue on as polygonal tilings until their intersection with the branches corresponding to \(4\pi/5\) and \(\pi/5\), respectively (see Fig.~\ref{fig:branchsamples} for the intersection of the \(\pi/5\) branch with the \(2\pi/3\) branch). Approaching this intersection, the shortest-period cells, the period-2 triangles (the large triangles for \(\gamma = 90^\circ\) and the four triangles at 2, 5, 8, and 10 o'clock positions at \(\gamma = 100^\circ\) in Fig.~\ref{fig:branchsamples}), are annihilated \cite{Smith2017} and transition to period-5 triangles at \(\gamma = 104.478...^\circ\), but remain polygonal tilings up to \(\gamma = 108^\circ\). Continuing along the \(2\pi/3\) branch, after the intersection with the \(\pi/5\) branch at \(\gamma = 108^\circ\), the period-5 triangles are retained as polygonal cells but not as a polygonal tiling (see for example the \(\tilde{E}\) structure at \(\gamma = 109^\circ\) along the \(2\pi/3\) branch in Fig.~\ref{fig:branchsamples}).
Similarly, moving upward from \(\gamma = 90^\circ\) along the \(\pi/5\) branch by increasing \(\gamma\), polygonal tilings with period-3 pentagons are retained until the intersection with the \(2\pi/3\) branch, after which the tiling disappears but the pentagons remain. These pentagons are retained until \(\gamma = 123.988...^\circ\) after which they disappear.
Thus, \((2j+1)\)-gons persist for a short distance past their intersection with the \(\pi/3\) and \(2\pi/3\) branches before disappearing along each branch even though the polygonal tilings do not (with the exception of the \(\pi/3\) and \(2\pi/3\) branches themselves before they intersect the \(4\pi/5\) and \(\pi/5\) branches, respectively).


\subsection{Other features in the symmetry plane}

The cathedral-like structure at the lower center of Fig.~\ref{fig:full-flame}(a), i.e.\ \(\gamma < 90^\circ\) and \(\alpha =\beta\) close to 90, is filled with the resonance curves that result when Arnold tongues intersect one of their symmetric mirror images (mirrored across \(\alpha + \beta = 180^\circ\)). For example, the intersections of short-dashed lines with long-dashed lines of the same color (blue with blue or black with black) in Fig.~\ref{fig:phasespaces}(d) lie on the same cathedral-like region in the \(\alpha + \beta = 180^\circ\) symmetry plane. Two symmetries of the system, specifically Eqs.~\ref{eq:symm-beta-gamma} and \ref{eq:symm-alpha-gamma}, reveal the symmetric equality of the \(\alpha + \beta = 180^\circ\) and \(\alpha = \beta\) symmetry planes. This allows us to connect the \(\alpha + \beta = 180^\circ\) plane shown for \(\gamma \ge 90^\circ\) in Fig.~\ref{fig:phasespaces},  which is where these particular tongue intersections occur, with  the \(\gamma \le 90^\circ\) region in Fig.~\ref{fig:full-flame}.


As noted earlier, 500 iterations is not sufficient to fully complete most of the structures of \(\tilde{E}\). This relatively low value is used to widen the resonant curves in this protocol space by not giving nearly resonant protocols the iterations necessary to fill in. This has the side-effect of creating the dark corners at the top of Fig.~\ref{fig:full-flame} that contain protocols for which \(\Phi_{\varepsilon,N}\) grows slowly but would approach \(\Phi = 1\) if \(N\) were large enough due to small atom sizes. Other regions in the protocol space have dark specks [most easily visible at the lower left corner and lower right corner of Fig.~\ref{fig:full-flame}(a)], which is also a result of protocols with at least one region of slowly growing coverage. These anomalies disappear at higher iteration counts, but using more iterations obscures the dark regions around resonances. Thus, apart from these specks and the dark corners, the overall structure of the patterns in the protocol space are not affected by the low number of iterations used in constructing Fig.~\ref{fig:full-flame}.

The remainder of the protocol space in Fig.~\ref{fig:full-flame}, along the \(\alpha = \beta\) symmetry plane but outside of the resonant curves, has complex structure that is not yet understood. For example, a set of prominent features in Fig.~\ref{fig:full-flame} are the dark fingers reaching up from branch intersections which correspond with the dark bands between Arnold tongue intersections mentioned earlier and which are prominent in Fig.~\ref{fig:phasespaces}(c). The protocol \((65.5^\circ, 65.5^\circ, 120^\circ)\) labeled \(A\) lies on one of these fingers. An analysis similar to that performed by Smith et al.\ \cite{Smith2017} shows that these dark fingers are the result of periodic points for large cells lying equidistant from two size-limiting cutting lines.


Although the focus here has been on the \(\alpha = \beta\) symmetry plane, there are many curves of resonant structures throughout the protocol space (i.e., for \(\alpha \neq \beta\)). However, this is beyond the scope of this study.

\subsection{Changing \(\gamma\) to increase average coverage}

There is an overall increase in coverage (lighter colors) as \(\gamma\) is increased or decreased from \(90^\circ\). Figure~\ref{fig:average_coverage} shows \(\Phi_{0.01,2\times 10^4}\) averaged across \( 0 \le \alpha,\beta \le 180^\circ\) as \(\gamma\) is varied. There is a clear minimum at \(\gamma = 90^\circ\), and \(\Phi\) generally grows away from \(\gamma = 90^\circ\). A sudden decrease near \(\gamma = 0\) and \(\gamma = 180^\circ\) where PWIs become rotations about a single axis is an artifact of sampling \(\Phi\) on a regular grid and using finite iterations. \(\Phi = 1\) for almost every protocol when \(\gamma = 0\) and \(\gamma = 180^\circ\) and the expected trend is shown with a dotted line. The overall trend agrees with the qualitative results in Fig.~\ref{fig:change-gamma} where, even though \(\Phi\) does not increase strictly monotonically (due to the existence of a polygonal tiling with zero coverage shown in Fig.~\ref{fig:57resonance}), the general trend of increasing coverage as \(\gamma\) is increased from \(90^\circ\) holds. Noting the scale on the \(\Phi\) axis, most protocols produce exceptional sets with fairly high coverage. This trend is apparent when comparing the average intensity as \(\gamma\) is increased in Fig.~\ref{fig:phasespaces}. One should note that this trend does not hold for the \(\alpha = \beta\) plane, Fig.~\ref{fig:full-flame}(a), in isolation. The local maxima around \(\gamma = 71^\circ\) and \(\gamma = 109^\circ\) are not well understood but may have a connection to the intersection of polygonal resonant branches that occurs at \(\gamma = 72^\circ\) and \(\gamma = 108^\circ\).

\begin{figure}
  \includegraphics[width = 0.5\textwidth]{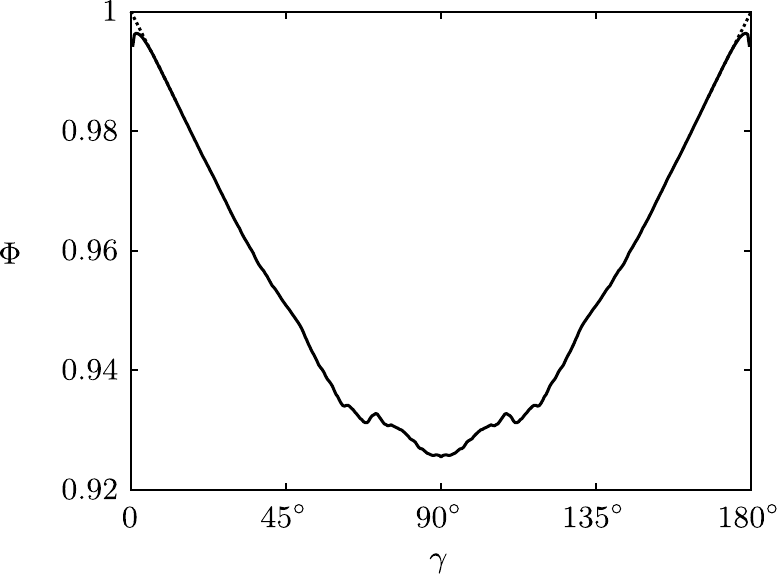}
  \caption{\(\Phi_{0.01,2\times 10^4}\) averaged across \( 0 \le \alpha,\beta \le 180^\circ\) vs. \(\gamma\). \(\Phi\) increases on average for axis arrangements further from orthogonal. 
  }
  \label{fig:average_coverage}
\end{figure}

\subsection{Atom size and shape determine Arnold tongue locations}

At a more fundamental level, when the areas of atoms (Eqs.~\ref{eq:area_p1}--\ref{eq:area_p4}) are compared, the corresponding curves along which two atoms are of equal area match the locations of the Arnold tongues originating from \((90^\circ, 0, \gamma)\) and \((0,90^\circ, \gamma)\) [short-dashed blue curves in Fig.~\ref{fig:phasespaces}(b-f)] or \((90^\circ, 180^\circ, \gamma)\) and \((180^\circ,90^\circ, \gamma)\) [long-dashed blue curves in Fig.~\ref{fig:phasespaces}(d)] exactly. Furthermore, these equal area curves match the blue dashed lines in Fig.~\ref{fig:phasespaces}(a) when \(\gamma = 90^\circ\). For example, for \(\gamma = 120^\circ\), these curves of equal atom size are shown in Fig.~\ref{fig:equal_atom_size_120}. Atoms \(P_1\) and \(P_4\) are only of equal size when \(\alpha + \beta = \pi\). Likewise, atoms \(P_2\) and \(P_3\) are only of equal size when \(\alpha = \beta\). The other lines correspond to \(A_2 = A_1\) (\(\alpha = \gamma'\))  and \(A_2 = A_4\) (\(\beta = \pi - \gamma'\)) in blue and \(A_3 = A_1\) (\(\beta = \gamma'\)) and \(A_3 = A_4\) (\(\alpha = \pi - \gamma'\)) in red. The smallest atom for each protocol, which often limits the mixing \cite{Smith2017}, is labeled in each sub region of Fig.~\ref{fig:equal_atom_size_120}. For orthogonal axes, there are no regions where \(P_2\) or \(P_3\) are the smallest atoms. This raises the question as to how much influence diversity in atom size has on mixing and how much information about mixing can be gained simply by examining the size and shape of the PWI atoms, which is a topic for future work.

\begin{figure}
  \includegraphics[width = 0.5\textwidth]{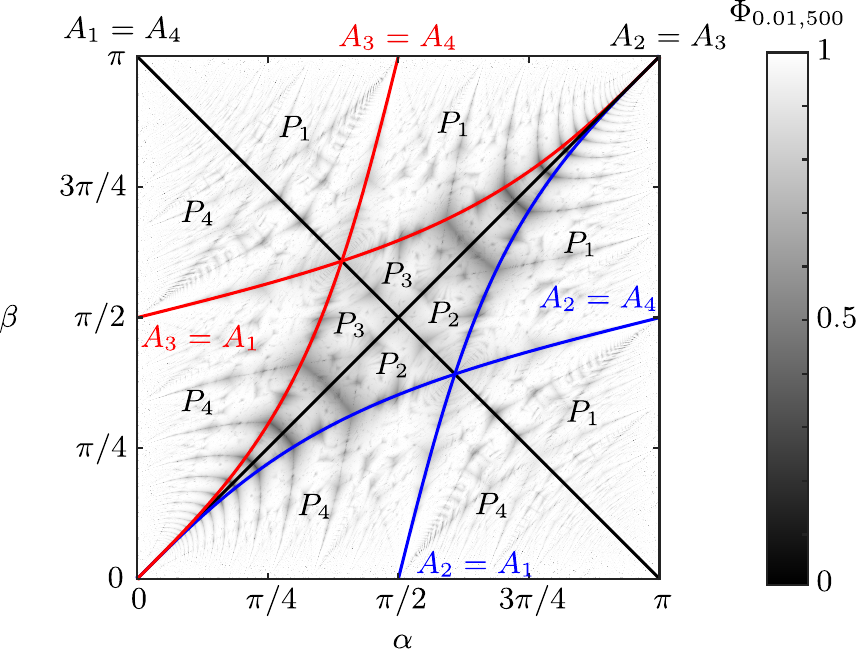}
  \caption{Curves of equal atom size (Eqs.~\ref{eq:area_p1}--\ref{eq:area_p4}) for \(M_\abg\) where \(\gamma = 120^\circ\). The smallest atom, \(P_i\), is labeled inside of each region between curves.
  }
  \label{fig:equal_atom_size_120}
\end{figure}

\section{Conclusions \label{sec:conclusions}}

Introducing non-orthogonal axes to the PWI on a hemispherical shell adds a new dimension to the parameter space that greatly expands possible cutting-and-shuffling behaviors. To explore this vast domain, a method for rapid generation of an approximation of the exceptional set was developed. The approach described here is computationally efficient for use with parallel computing. Non-orthogonal axes expand the regions in parameter space where certain periodic points (responsible for non-mixing islands) can exist while allowing for new periodic points that are not possible in the orthogonal axes case. Some of the protocols with non-orthogonal axes have more coverage than their corresponding orthogonal axes protocols while others have less. On average, coverage of the hemispherical shell by the exceptional set is increased by using non-orthogonal axes.

When examining the protocol space for the fractional coverage of \(\tilde{E}\), various branches of locally resonant protocols traverse the protocol space. The most prominent appear along symmetry planes, one of which, \(\alpha = \beta\), is investigated thoroughly and contains many protocols with polygonal cells, including protocols whose cells form a polygonal tiling of the entire hemispherical shell. Arnold tongues in the protocol space identify low coverage protocol regions. Intersections of symmetric pairs of Arnold tongues (along symmetry planes) indicate locations of the polygonal exceptional sets. Some symmetries found exclusively in the orthogonal axes case are broken when considering non-orthogonal axes allowing increased transport and a wide variety of unique PWIs.

This paper only just hints at the possible behaviors for non-orthogonal axes. For example, an investigation of the relationship between atom size and shape following up on the cursory results shown in Fig.~\ref{fig:equal_atom_size_120} could shed more light on the cause of low coverage resonances. An extension of this PWI to the entire unit sphere, eliminating periodic boundaries, outlined in Appendix~\ref{sec:extended-pwi}, could also be a fruitful topic for future study. Additionally, the observation  that coverage of the exceptional set \(\Phi\) does not always correlate well with mixing, such as in Fig.~\ref{fig:tax_orthog}, leads to questions about barriers to mixing within the exceptional set itself. Similar barriers to transport within the exceptional set have been observed by Ashwin et al. \cite{Ashwin2018}. However, their occurrence and dynamics have not been fully explained yet.

\begin{acknowledgments}
This material is based upon work supported by the National Science Foundation under grant no.\ CMMI-1435065.
\end{acknowledgments}

\appendix
\section{Extending the PWI mapping to the entire sphere \label{sec:extended-pwi}}
The PWI here and previously \cite{Park2016,Park2017,Smith2017} is restricted the domain to the hemispherical shell, but this is an unnecessary restriction. The behavior of this PWI can be easily expanded to the full sphere by simply changing rotation direction based on the sign of the \(y\)-value of the point before rotation about the \(y\)-axis. This change brings the PWI system very close in form to a simpler system studied previously \cite{Scott2001,Scott2003} and allows us to do away with modular rotations or additional rotations by \(\pi\) and instead produce a compact and complete view of the PWI over the entire sphere. This system maintains all symmetries derived in Appendix~\ref{sec:symmetries}.

First, we rewrite the rotation procedure in a slightly different format that is mathematically identical to Eq.~\ref{eq:mapping} \cite{Smith2017}. After rotation by \(\alpha\) about the \(z\)-axis, rotate the HS by \(-\gamma\) about the \(y\)-axis to align the second axis with the \(z\)-axis before the second rotation by \(\beta\). Rotate by \(\beta\) about the \(z\)-axis and then rotate about the \(y\)-axis by \(\gamma\) to restore the original axial positions. The mapping can then be written as
\begin{equation}
  M_\abg = R^y_{\gamma} \tilde{M}^z_{\beta} (R^y_{\gamma})^{-1} \tilde{M}^z_\alpha, \label{eq:gamma_pwi}
\end{equation}
where \(R^{a}_{\theta}\) is a rotation about axis \(a\) by \(\theta\) and \(\tilde{M}\) is a ``modular'' rotation (points rotated above the equator, \(y > 0\), are rotated by an additional \(\pi\)) such that \(\tilde{M}^z_{\theta + n\pi} = \tilde{M}^z_{\theta}\).

Between any two modular flips where the system rotates by \(R^z_{\pi}\), the equivalent full-sphere system has inverted \(y\) rotations and unchanged \(z\) rotations. This is easily shown by considering the compositions of the modular flip rotations, \(R_\pi^z\), with rotations about the \(z\)- and \(y\)-axes,
\begin{eqnarray}
  R^z_{\pi} R^z_{\theta} &=& R^z_{\theta + \pi} \nonumber \\
  &=& R^z_{\theta} R^z_{\pi} \label{eq:z_pass_z}\\
  R^z_{\pi} R^y_{\theta} &=& (S_{yz}S_{xz}) R^y_{\theta} \nonumber \\
  &=& S_{yz} R^y_{\theta} S_{xz} \nonumber \\
  &=& (R^y_{\theta})^{-1} S_{yz}S_{xz} \nonumber \\
  &=& (R^y_{\theta})^{-1} R^z_{\pi}. \label{eq:z_pass_y}
\end{eqnarray}
Equation~\ref{eq:z_pass_z} shows that the modular flips, \(R_\pi^z\), commute with rotations about the \(z\)-axis. Equation~\ref{eq:z_pass_y} shows that interchanging the order of a rotation about the \(y\)-axis with \(R_\pi^z\) is possible, but inverts the \(y\)-axis rotation. Using this, any two modular flips can annihilate one another by shifting them together in the rotation order until they collide, inverting \(y\)-axis rotations as they move.\footnote{If there is an odd number of modular flips for a given point, there will be an unpaired \(R_\pi^z\) left over indicating that the final position of the point lies on the opposite hemisphere (upper or lower) than its starting position and would need to be flipped over one more time to return to the starting hemisphere. We ignore this correction and allow points to end up on the opposite hemisphere since our goal is to extend the PWI to the full sphere.} Rotations that were between pairs of \(R_\pi^z\) (inverting \(y\)-axis rotations due to Eq.~\ref{eq:z_pass_y}) represent actions on points that remain in the upper hemisphere (i.e.\ \(y>0\)) which is outside of the lower hemispherical domain used in this paper and earlier \cite{Park2016,Park2017,Smith2017}. Points that remain in the lower hemisphere have unchanged \(y\)-axis rotations. As such, the signum function can be used on the \(y\) coordinate to dictate which direction rotations about the \(y\)-axis should occur.

The new PWI system is as follows:
\begin{equation}
  M_\abg  = (R^{y}_{\gamma \, \sign(y)})^{-1} R^z_\beta R^{y}_{\gamma \, \sign(y)} R^z_\alpha,
\end{equation}
where \(\sign(y)\) is evaluated upon application of that particular rotation, and the signum function is defined here as
\begin{equation}
  \sign (y) = 
  \left\{
    \begin{array}{rl}
      1, & y \ge 0 \\
      -1, & y \le 0
    \end{array}
  \right. \nonumber
\end{equation}
such that the rotation about \(y\) provides cutting behavior when the rotated point falls onto the equator, which was \(\partial S\) for the hemisphere, and the signum function is multi-valued at \(y=0\), mirroring the multi-valued nature of \(\partial S\) in \(M_\abg\). The treatment of \(y=0\) in the signum function changes dynamics within \(E\), but does not affect the behavior outside of \(E\) including \(\bar{E} \setminus E\). Treating \(y=0\) as a multi-valued discontinuous set maintains dynamics within \(E\) although dynamics within \(E\) are not considered in this paper.

This new system can be described as rotation about two axes where in between rotation about an axis, the top and bottom hemispheres are twisted about the \(y\)-axis in opposite directions. The equator in this system still acts as a cutting line in this regard so the system will maintain identical structures for \(E\) but copied on both hemispheres and with doubled periods for some periodic points.

When formulated this way, it is easy to see that the system studied by Scott et al. \cite{Scott2001, Scott2003} is equivalent to a hemispherical shell with periodic boundary that is rotated by a constant amount, \(\omega\), about an axis of rotation that itself rotates in the equatorial plane after each rotation by a constant amount, \(\mu\).

\section{Proof of finite \(N\) coverage of \(\bar{E}\) \label{sec:fat_coverage_proof}}

\noindent \textbf{Claim:} For all \(\varepsilon>0\), and protocol parameters \(\abg\), there exists \(N=N(\varepsilon,\abg)<\infty\) such that 
\begin{equation}
E \subset \bar{E} \subset \tilde{E}_{\varepsilon,N} = \bigcup_{n=0}^N M_{\alpha,\beta,\gamma}^n \mathcal{D}_\varepsilon \subset E_\varepsilon.
\end{equation}
This is the claim made in Eq.~\ref{eq:fat_coverage}. We use \(M\) as a shorthand for \(M_\abg\) for brevity here although this result is valid for any invertible PWI.

\begin{proof}
Fix \(\delta > 0\), and let \(P\) be a finite collection of \(\delta\)-balls, centered on points in \(S\), whose union contains \(S\), i.e.\ \(P=\{ B_\delta(\bm{x}) : \bm{x}\in S\}\), \(|P|<\infty\), and \(\bigcup_{U\in P} U \supset S\).

Let \(Q=\{U\in P : U\cap \bar{E} \neq \emptyset\}\), so \(|Q|<\infty\) and \(\bigcup_{U\in Q} U \supset \bar{E}\).  

Take any \(\bm{x} \in \bar{E}\), then there exists \(U=B_\delta(\bm{u})\in Q\) such that \(\bm{x} \in U\). We will show that \(\bm{x} \in M^{N_U} \mathcal{D}_\varepsilon\) for some \(N_U < \infty\) \footnote{This is similar in spirit to Poincar\'e recurrence, in the sense that points in \(E_\varepsilon\) will return to within \(\varepsilon\) of \(\mathcal{D}\), i.e.\ \(\mathcal{D}_\varepsilon\), in a finite number of iterations.}. Suppose that \(U\cap E \neq \emptyset\). Let \(\bm{y} \in E\) be the first element of \(U\) that is `hit' by some element of \(\mathcal{D}\), i.e.\ there exists \(N_U\) such that \(M^{-N_U}(\bm{y}) \in \mathcal{D}\), and \(M^{-i}(\bm{a}) \notin \mathcal{D}\) for all \(\bm{a} \in U\) and \(0\leq i < N_U\). The last statement is equivalent to saying that \(M^{-i} U \cap \mathcal{D} = \emptyset\) for all \(0\leq i < N_U\). Therefore, mapping backward in time for \(N_U\) iterations, the set \(U\) does not intersect a cutting line, and so will not be cut. Furthermore, aside from the cuts, the map \(M\) is an isometry, meaning it preserves distances. Now let \(\delta = \varepsilon/2\). Since \(\bm{x},\bm{y} \in U = B_{\varepsilon/2}(\bm{u})\), \(\bm{x}\) and \(\bm{y}\) can be at most \(\varepsilon\) apart, i.e.\ \(\|\bm{x}-\bm{y}\|<\varepsilon\). Since \(M\) is an isometry, \(\|M^{-N_U}(\bm{x}) - M^{-N_U}(\bm{y}) \|=\|\bm{x}-\bm{y}\|<\varepsilon\), so \(M^{-N_U}(\bm{x}) \in B_\varepsilon(M^{-N_U}(\bm{y}))\subset \mathcal{D}_\varepsilon\) and hence \(\bm{x} \in M^{N_U}(\mathcal{D}_\varepsilon)\). Therefore, \(\bm{x} \in M^{N_U} \mathcal{D}_\varepsilon\) as desired. 

Suppose \(U\cap E = \emptyset\), i.e.\ \(\bm{x}\) is a limit point, not an interior point, of \(E\). Let \(r=\|\bm{x}-\bm{u}\|<\delta\). Since \(\bm{x}\) is a limit point of \(E\), for all \(\eta>0\) there exists a point \(\bm{y}\in E\cap B_\eta (\bm{x})\). Specifically, for \(\eta=\delta-r\), there exists \(\bm{y}\in E\cap B_{\delta-r} (\bm{x})\). Therefore, \(\bm{y} \in E\) and
\begin{multline}
\|\bm{y} - \bm{u}\|= \|\bm{y} - \bm{x} + \bm{x} - \bm{u}\| \\ \leq \|\bm{y} - \bm{x}\| + \|\bm{x} - \bm{u}\| < (\delta-r) +r = \delta,
\end{multline}
i.e.\ \(\bm{y} \in U\cap E\), which contradicts \(U\cap E = \emptyset\). Therefore, the only possibility is that \(U\cap E \neq \emptyset\). 

Let \(N = \max_{U\in Q} N_U\), which is finite because \(|Q|<\infty\) and \(N_U<\infty\) for each \(U\in Q\). It follows that for all \(\bm{x} \in \bar{E}\), \(\bm{x} \in U\) for some \(U \in Q\), and \(\bm{x} \in M^{N_U} \mathcal{D}_\varepsilon \subset  \bigcup_{n=0}^N M_{\alpha,\beta,\gamma}^n \mathcal{D}_\varepsilon\).
\end{proof}

\section{Symmetries in the protocol space of the \(M_\abg\) mapping \label{sec:symmetries}}

Following the general analysis of Smith et al.\ \cite{Smith2017}, the mapping \(M_\abg\) in Eq.~\ref{eq:mapping} can be rewritten using only modular rotations about the \(z\)-axis and rotations about the vertical \(y\)-axis (refer to Fig.~\ref{fig:pwi-demo} for an example of the relevant coordinate system) to more easily find symmetries within the protocol space of the BST PWI based on symmetries of single axis modular rotation. This has been done above in Eq.~\ref{eq:gamma_pwi}.

Since rotations about the \(z\)-axis are modular, they are periodic in the protocol space with period \(\pi\) and all unique behavior is captured in the region \( 0 \le \alpha \le \pi\) and \(0\le \beta \le \pi\). Likewise, rotation about the \(y\)-axis is periodic in the protocol space with period \(2\pi\) and all unique behavior in the protocol space is within \(0\le \gamma \le 2\pi\). Note that \((R^y_\theta)^{-1} =R^{-y}_\theta = R^y_{-\theta} \) are all equivalent inverses of \(R^y_\theta\).

Taking the labels \(i,j,k\) to represent a set of orthogonal axes (e.g., \(x,y,z\)), the following key identities for rotations about a single axis will be extended to the PWI mapping,
\begin{align}
  R^i_\theta &= S_{jk}R^i_\theta S_{jk}, \label{eq:reflection-constant-symm}\\
  R^i_\theta &= S_{ij} (R^i_\theta)^{-1} S_{ij}, \label{eq:reflection-reversal-symm}\\
  S_{ij}S_{ik} = S_{ik}S_{ij} &= R^i_\pi = R^i_{-\pi}, \label{eq:reflection-pi}
\end{align}
where \(S_{ij}\) is a reflection across the \(ij\)-plane (e.g.~\(S_{xy}\) reflects across the \(xy\)-plane, \(z\mapsto -z\)). Equation~\ref{eq:reflection-constant-symm} is a rotation that is unchanged under reflection across the plane orthogonal to the rotation axis , it also provides that the rotation \(R_\theta^i\) and reflection \(S_{jk}\) commute. Equation~\ref{eq:reflection-reversal-symm} is a reflection-reversal symmetry where the rotation is reversed under reflection across a plane containing the rotation axis. Equation~\ref{eq:reflection-pi} simply states that reflection about two perpendicular planes is identical to rotation by \(\pi\) about the axis formed by the planes' intersection.

These identities can also be trivially extended to a few corollaries, Eqs.~\ref{eq:corollary1}--\ref{eq:corollary4}, which will be used liberally in the derivation of PWI symmetries,
\begin{align}
  S_{jk} R^i_\theta &= S_{jk} R^i_\theta (S_{jk} S_{jk}) = R^i_\theta S_{jk}, \label{eq:corollary1} \\
  S_{ij} R^i_\theta &= S_{ij} R^i_\theta (S_{ij} S_{ij}) = (R^i_\theta)^{-1} S_{ij}, \\
  S_{ij} &= S_{ij} (S_{ik} S_{ik}) = R^i_\pi S_{ik}, \\
  S_{ij} &= S_{ij} (S_{jk} S_{jk}) = R^j_\pi S_{jk}. \label{eq:corollary4} 
\end{align}

We wish to first extend these identities to the single modular rotation operator \(\tilde{M}^z_\theta\). For most identities, this is trivial, except for the following,
\begin{equation}
  \tilde{M}^z_\theta = S_{xz}(\tilde{M}^z_\theta)^{-1} S_{xz}.
\end{equation}
\(S_{xz}\) maps the lower hemisphere to the upper hemisphere, where \(\tilde{M}\) is undefined. \(\tilde{M}\) could be carefully redefined to provide an additional rotation by \(\pi\) whenever a point crosses the equator (i.e., providing an additional rotation by \(\pi\) when moving from the upper to the lower hemisphere in addition to when moving from the lower to the upper) instead of the simple modular operation at the equator to allow the use of symmetries involving \(S_{xz}\). For the sake of clarity, this identity will not be used (and is not needed) in any of the following derivations.

\begin{figure}
  \includegraphics[width = 0.45\textwidth]{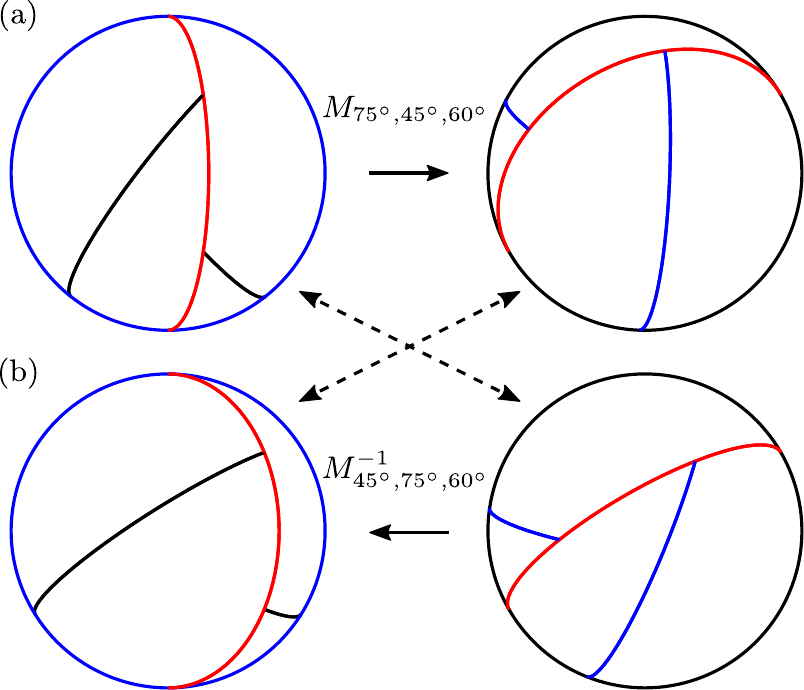}
  \caption{(a) PWI for \(M_{75^\circ,45^\circ,60^\circ}\). (b) PWI for \(M_{45^\circ,75^\circ,60^\circ}\) with the inverse direction indicated (the initial condition for the inverse map is on the right). Compare the starting (or ending) positions of \(M_{75^\circ,45^\circ,60^\circ}\) and \(M^{-1}_{45^\circ,75^\circ,60^\circ}\) (marked with the dashed arrows) to see symmetry between the two protocols indicated in Eq.~\ref{eq:symm-reverse-time}.}
  \label{fig:symm-reverse-time}
\end{figure}

Symmetry between \(M_\abg\) and \(M^{-1}_{\beta,\alpha,\gamma}\) can be seen by applying Eq.~\ref{eq:reflection-reversal-symm} to both \(\tilde{M}^z\) and then \(R^y\),
\begin{align}
  M_\abg
  &= R^y_\gamma \tilde{M}^z_\beta R^{-y}_\gamma \tilde{M}^z_\alpha \nonumber \\
  &= R^y_\gamma (S_{yz} \tilde{M}^{-z}_\beta S_{yz}) R^{-y}_\gamma (S_{yz} \tilde{M}^{-z}_\alpha S_{yz}) \nonumber \\
  &= (R^y_\gamma S_{yz}) M^{-1}_{\beta,\alpha,\gamma} (R^y_\gamma S_{yz}). \label{eq:symm-reverse-time}
\end{align}
Thus there is a symmetry about \(\alpha = \beta\) for the forward and reverse mappings for all \(\gamma\). This means that \(\bar{E}\) and the locations of cells (both of which are invariant) will be identical under \(M_\abg\) and \(M_{\beta,\alpha,\gamma}\) up to the reflection \(R^y_\gamma S_{yz}\), an involution. For all of these symmetries, it can shown that the reflections and rotations that bookend the mappings on either side are all involutions indicating these symmetries hold for any number of iterations. Therefore, the fractional coverage of \(\bar{E}\), \(\Phi\), for both \(M_\abg\) and \(M_{\beta,\alpha,\gamma}\) are equivalent. Thus, there is, in the protocol space explored in Section~\ref{sec:phase_space}, a symmetry across \(\alpha = \beta\) (\(\alpha - \beta = n\pi\) when the periodicity rotations about the \(z\)-axis is taken into account). As a result, this reflection symmetry in the protocol space restricts the protocols for which there is unique mixing behavior to those where either \(\alpha \le \beta\) or \(\beta \le \alpha\). An example case of this symmetry in shown in Fig.~\ref{fig:symm-reverse-time} where \(M_{75^\circ,45^\circ,60^\circ}\) and \(M^{-1}_{45^\circ,75^\circ,60^\circ}\) are compared. The initial (and final) positions of \(M_{75^\circ,45^\circ,60^\circ}\) and \(M^{-1}_{45^\circ,75^\circ,60^\circ}\) differ by a rotation and a reflection as expected.

Equation~\ref{eq:reflection-reversal-symm} can be applied again, with slight variation, to get a symmetry between \(M_\abg\) and \(M_{-\alpha,-\beta,-\gamma}\) [shown between Fig.~\ref{fig:combined_symmetries}(a) and (b)],
\begin{align}
  M_\abg
  &= R^y_\gamma \tilde{M}^z_\beta R^{-y}_\gamma \tilde{M}^z_\alpha \nonumber \\
  &= (S_{yz} S_{yz}) R^y_\gamma (S_{yz} \tilde{M}^{-z}_\beta S_{yz}) R^{-y}_\gamma (S_{yz} \tilde{M}^{-z}_\alpha S_{yz}) \nonumber \\
  &= S_{yz} M_{-\alpha,-\beta,-\gamma} S_{yz}. \label{eq:symm-negate-all}
\end{align}

Noting that rotation about the \(y\)-axis is periodic with period \(2\pi\) and applying Eq.~\ref{eq:reflection-constant-symm} reveals a symmetry across \(\gamma = n\pi\) [shown between Fig.~\ref{fig:combined_symmetries}(a) and (c)],
\begin{align}
  M_\abg
  &= R^y_\gamma \tilde{M}^z_\beta R^{-y}_\gamma \tilde{M}^z_\alpha \nonumber \\
  &= (S_{xy} S_{xy}) R^y_\gamma (S_{xy} \tilde{M}^{z}_\beta S_{xy}) R^{-y}_\gamma (S_{xy} \tilde{M}^{z}_\alpha S_{xy}) \nonumber \\
  &= S_{xy} M_{\alpha,\beta,-\gamma} S_{xy}. \label{eq:symm-negate-gamma}
\end{align}
Since behavior is reflected across \(\gamma = \pi\) in the protocol space, the protocols producing unique PWIs are further restricted to \(0 \le \gamma \le \pi\).

Applying both Eqs.~\ref{eq:symm-negate-all} and \ref{eq:symm-negate-gamma} demonstrates the symmetry about \(\alpha = \beta = 0\) in the protocol space [shown between Fig.~\ref{fig:combined_symmetries}(a) and (d)],
\begin{align}
  M_\abg
  &= S_{xy} M_{\alpha,\beta,-\gamma} S_{xy} \nonumber \\
  &= S_{xy} ( S_{yz} M_{-\alpha,-\beta,\gamma} S_{yz} ) S_{xy} \nonumber \\
  &= R^y_\pi M_{-\alpha,-\beta,\gamma} R^y_\pi. \label{eq:symm-negate-two}
\end{align}
Noting that the rotations about the \(z\)-axis are periodic with period \(\pi\) and there is an existing reflection symmetry across \(\alpha - \beta = 0\) (Eq.~\ref{eq:symm-reverse-time}),
this is a reflection symmetry across \(\alpha + \beta = \pi\) in the protocol space. The region containing unique PWI behavior is then further reduced by this symmetry plane to the wedge \( 0 \le \alpha \le \pi/2\), \(0 \le \beta \le \min(\alpha, \pi - \alpha)\), and \(0 \le \gamma \le \pi\).

Equation~\ref{eq:reflection-reversal-symm} can be applied to a single rotation to show the symmetry relating \(\beta\) and \(\gamma\) [shown in Fig.~\ref{fig:combined_symmetries}(a)],
\begin{align}
  M_\abg
  &= R^y_\gamma \tilde{M}^z_\beta R^{-y}_\gamma \tilde{M}^z_\alpha \nonumber \\
  &= R^y_\gamma [S_{yz} (S_{xy} S_{xy}) \tilde{M}^{-z}_\beta (S_{xy} S_{xy}) S_{yz}] R^{-y}_\gamma \tilde{M}^z_\alpha \nonumber \\
  &= M_{\alpha,-\beta,\gamma - \pi}. \label{eq:symm-beta-gamma}
\end{align}
This symmetry combined with Eq.~\ref{eq:symm-negate-gamma} indicates a symmetry about \(\beta = n\pi/2\) and \(\gamma = \pi - \gamma + 2m\pi\) or \(\gamma = (m+\frac{1}{2})\pi\) where reflecting across \(\beta = \pi/2\) requires a reflection across \(\gamma = \pi/2\) as well, indicating a rotational symmetry (rotation by \(\pi\)) in the protocol space about \((\alpha, \pi/2, \pi/2)\). For the case when \(\gamma = \pi/2\), this is a reflection symmetry across \(\beta = \pi/2\) previously shown in Fig.~\ref{fig:phasespaces}(a) as a dashed red line. For \(\gamma = \pi/2\), this reduces the unique PWI behaviors to the triangular region \(0 \le \alpha \le \pi/2\) and \(0 \le \beta \le \alpha\). Extending to all values of \(\gamma\), this indicates that the PWI for \(\gamma \le \pi/2\) are rotationally symmetric in the protocol space to other PWI for \(\pi/2 \le \gamma\). This symmetry splits the remaining space for unique PWI in half through any plane that passes through \(\alpha, \pi/2, \pi/2\).

Equation~\ref{eq:symm-beta-gamma} can be extended to the \(\alpha\) direction by applying Eqs.~\ref{eq:symm-reverse-time} and \ref{eq:symm-beta-gamma} carefully [shown between Fig.~\ref{fig:combined_symmetries}(a) and (d)],
\begin{align}
  M_\abg
  &= R^y_\pi M_{-\alpha,\beta,\gamma-\pi} R^y_{\pi}. \label{eq:symm-alpha-gamma}
\end{align}
Similar to the argument following Eq.~\ref{eq:symm-beta-gamma}, this results in a rotational symmetry due to two reflection symmetries across \(\alpha = n\pi/2\) and \(\gamma = (m+\frac{1}{2})\pi\).

\begin{figure}
  \includegraphics[width = 0.45\textwidth]{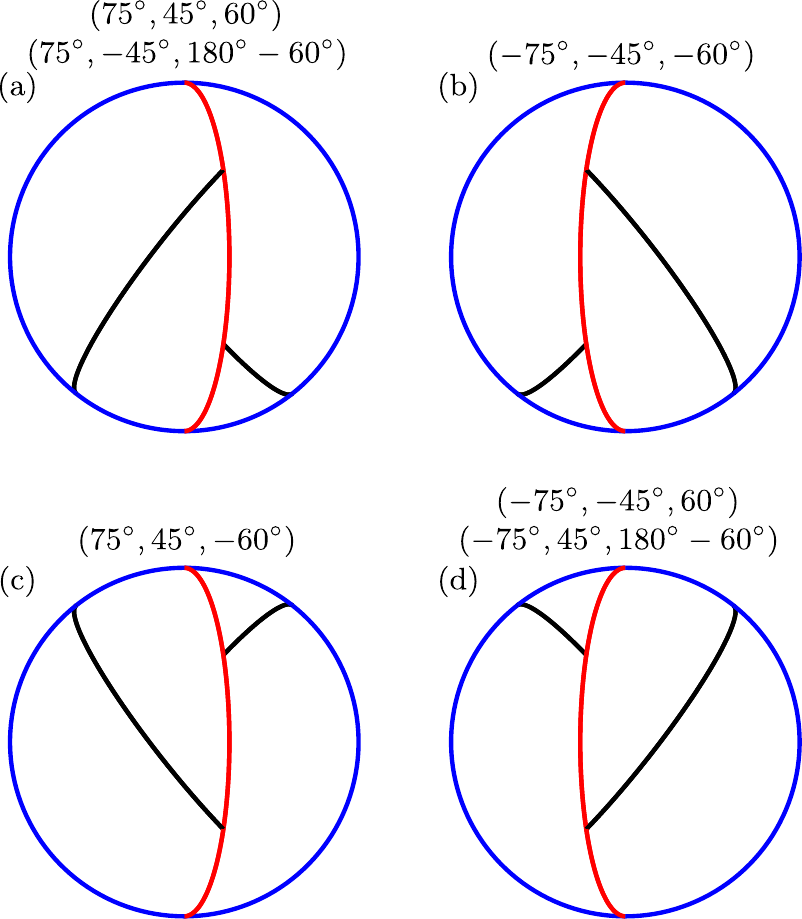}
  \caption{The initial conditions of symmetric protocols. (a) \((75^\circ,45^\circ,60^\circ)\) is identical to \((75^\circ,-45^\circ,180^\circ -60^\circ)\) by Eq.~\ref{eq:symm-beta-gamma}. (b) \((-75^\circ,-45^\circ,-60^\circ)\) is symmetric (flipped horizontally) to (a) by Eq.~\ref{eq:symm-negate-all}. (c) \((75^\circ,45^\circ,-60^\circ)\) is symmetric (flipped vertically) to (a) by Eq.~\ref{eq:symm-negate-gamma}. (d) \((-75^\circ,-45^\circ,60^\circ)\) is identical to \((-75^\circ,45^\circ,180^\circ - 60^\circ)\) both of which are symmetric (rotated about the \(y\)-axis by \(\pi\)) to (a) by Eqs.~\ref{eq:symm-negate-two} and \ref{eq:symm-alpha-gamma}.}
  \label{fig:combined_symmetries}
\end{figure}

One consequence of these symmetry planes in the protocol space shows that applying \(\pi - \theta\) for any two angles, \(\alpha, \beta,\gamma\), will maintain the same structure of \(E\) up to some rotation or reflection.

Additionally, one can examine \(M^2_\abg\) to see further similarity of structure between \(M_\abg\) and \(M_{\beta,\alpha,\gamma}\) and consequently \(M^{-1}_{\alpha, \beta, \gamma}\),
\begin{align}
  M^2_\abg
  &= (R^y_\gamma \tilde{M}^z_\beta R^{-y}_\gamma \tilde{M}^z_\alpha)^2 \nonumber \\
  &= (S_{xy} R^{-y}_\gamma \tilde{M}^z_\beta R^{y}_\gamma \tilde{M}^z_\alpha S_{xy})^2 \nonumber \\
  &= (S_{xy} R^{-y}_\gamma \tilde{M}^z_\beta) M_{\beta,\alpha,\gamma} (R^{y}_\gamma \tilde{M}^z_\alpha S_{xy}) \nonumber \\
  &= (S_{xy} R^{-y}_\gamma \tilde{M}^z_\beta) (R^y_\gamma S_{yz}) M^{-1}_\abg (R^y_\gamma S_{yz}) (R^{y}_\gamma \tilde{M}^z_\alpha S_{xy})  \nonumber \\
  &= L_1 M^{-1}_\abg L_2.    \label{eq:secondorder-inverse}
\end{align}
where \(L_1 = (S_{xy} R^{-y}_\gamma \tilde{M}^z_\beta) (R^y_\gamma S_{yz})\) is a single cut-and-shuffle about a reoriented horizontal axis and \(L_2 = (R^y_\gamma S_{yz}) (R^{y}_\gamma \tilde{M}^z_\alpha S_{xy})\) is also a single cut-and-shuffle about a reoriented horizontal axis.

The transformations \(L_1\) and \(L_2\) satisfy
\begin{align}
    L_2 L_1
  	&=  (R^y_\gamma S_{yz}) (R^{y}_\gamma \tilde{M}^z_\alpha S_{xy}) (S_{xy} R^{-y}_\gamma \tilde{M}^z_\beta) (R^y_\gamma S_{yz})  \nonumber \\
  	&= (R^y_\gamma S_{yz}) (R^{y}_\gamma \tilde{M}^z_\alpha R^{-y}_\gamma \tilde{M}^z_\beta) (R^y_\gamma S_{yz}) \nonumber \\ 
  	&= (R^y_\gamma S_{yz}) M_{\beta,\alpha,\gamma} (R^y_\gamma S_{yz})  \nonumber \\ 
  	&=  M^{-1}_\abg, \nonumber
\end{align}
such that when the right hand side of Eq.~\ref{eq:secondorder-inverse} is composed with itself, it yields
\begin{equation}
L_1 M^{-1}_\abg L_2L_1 M^{-1}_\abg L_2 = L_1 M^{-3}_\abg L_2. \nonumber
\end{equation}
Hence, there is a relationship between the forward and reverse mappings,
\begin{equation}
  M^{2n}_\abg = L_1 M^{-2n + 1}_\abg L_2. \label{eq:plus_minus_equiv}
\end{equation}
This relationship reveals that the inverse mapping is simply a cut-and-shuffle of the forward mapping and thus forward iterates of \(M_\abg\) generate the same cutting lines as reverse iterates separated by a single cut-and-shuffle. This can be understood by noting that cell locations are invariant in the forward and reverse mappings, and \(\bar{E}\), a collection of cutting lines and their limit points, is the complement of regular set containing all the periodic cells and therefore also invariant under \(M_\abg\).

Due to the above symmetries in structure, the complete set of unique PWI behaviors of this type exists in a small region between symmetry planes. Based on the periodicity of the rotations, all behaviors can be found in the region \(0 \le \alpha < \pi\), \(0 \le \beta < \pi\), \(0 \le \gamma < 2\pi\). Symmetry can be used to reduce the region of protocol space investigated  to only that containing unique behavior, specifically the region \( 0 \le \alpha \le \pi\), \(0 \le \beta \le \min(\alpha, \pi - \alpha)\), \(0 \le \gamma \le \frac{\pi}{2}\) or the alternative region \( 0 \le \alpha \le \frac{\pi}{2}\), \(0 \le \beta \le \alpha\), \(0 \le \gamma \le \pi\) between symmetry planes (depending on how the rotational plane through \((\alpha, \pi/2,\pi/2)\) from Eq.~\ref{eq:symm-beta-gamma} is chosen). Despite this, much more of the protocol space has been shown in Fig.~\ref{fig:phasespaces} to highlight these symmetries. A diagram showing the planes of symmetry in 3D protocol space is shown in Fig.~\ref{fig:protocol_space}.

\begin{figure}
  \includegraphics[width = 0.45\textwidth]{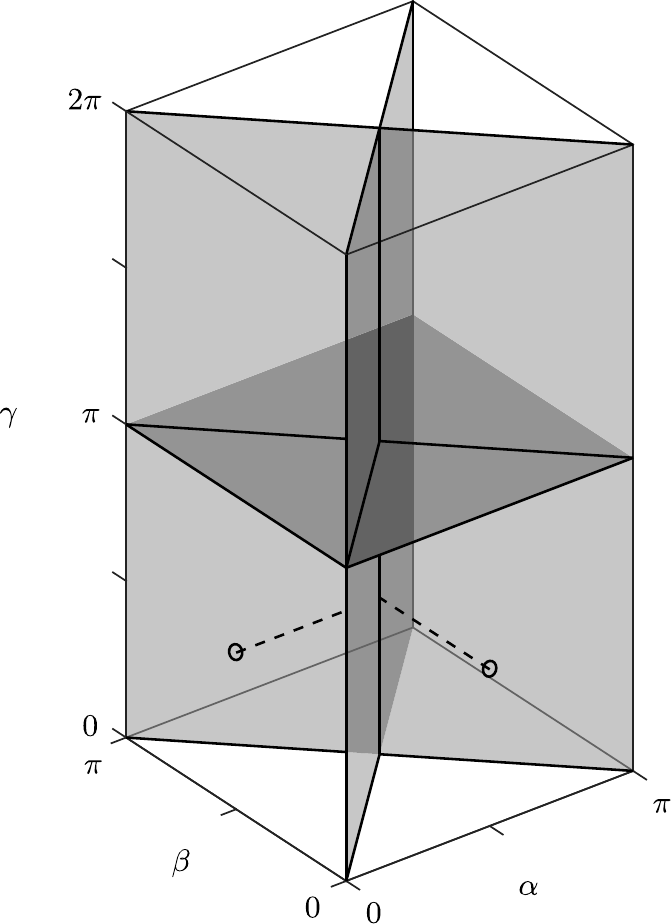}
  \caption{Symmetry planes in the PWI protocol space (bounds of this space are periodic). The plane \(\alpha = \beta\) is due to Eq.~\ref{eq:symm-reverse-time}. The plane \(\alpha = \pi - \beta\) is due to Eq.~\ref{eq:symm-negate-two}. The plane \(\gamma = \pi\) is due to Eq.~\ref{eq:symm-negate-gamma}. The dashed lines \((\alpha,\pi/2,\pi/2)\) and \((\pi/2,\beta,\pi/2)\) are axes about which there is a rotational symmetry due to Eqs.~\ref{eq:symm-beta-gamma} and \ref{eq:symm-alpha-gamma}, respectively.}
  \label{fig:protocol_space}
\end{figure}

These symmetry derivations expand on those found in \cite{Smith2017} and indicate that the reflection symmetries across \(\alpha = \pi/2\) and \(\beta = \pi/2\) are only reflections for \(\gamma = \pi/2\).

\section{Lambert area-preserving azimuthal projection\label{sec:lambert}}
The Lambert azimuthal equal area projection is an area preserving mapping between a sphere and a flat disk \cite{Snyder1987}. For the purposes of this research, it is desirable to have a simple projection between the hemispherical shell (HS) and a planar grid which preserves area to allow easy measurement of computed region sizes. Using a unit sphere, \(S = \{\mathbf{x}\in \mathbb{R}^3 : \|\mathbf{x}\|_2 =1\}\), and a Cartesian coordinate system, \(\mathbf{x} = (x,y,z)\), the projection has a compact form. The mapping from the unit sphere to the plane when centered about the negative \(y\)-axis is given by,
\begin{equation}
  \label{eqn:lamberttoplane}
  (X,Z) = \left(\sqrt{\frac{2}{1-y}}x,\sqrt{\frac{2}{1-y}}z \right),
\end{equation}
and its corresponding inverse projection from the disk to the sphere is given by,
\begin{align}
  \label{eqn:lamberttosphere}
  \begin{pmatrix}
    x \\
    y \\
    z
  \end{pmatrix}
  = 
  \begin{pmatrix}
    \sqrt{1 - \frac{X^2 + Z^2}{4}}X \\
    -1 + \frac{X^2+Z^2}{2} \\
    \sqrt{1 - \frac{X^2 + Z^2}{4}}Z
  \end{pmatrix}.
\end{align}

The hemispherical shell is projected from the sphere to a disk in the \(XZ\)-plane centered at \((0,0)\) with radius \(\sqrt{2}\). For simplicity, a factor of \(\sqrt{2}\) is removed in calculations to produce a unit disk since relative areas on the sphere, not the absolute areas, are used in the computation of \(\Phi\).


If a regular Cartesian grid is used for the projection, the maximum distance between two points on the HS can be related to the Cartesian grid spacing by taking the limit of the distance between a point on the equator and a point diagonally adjacent in the grid as the grid resolution goes to \(\infty\). Given a grid spacing of \(\frac{2}{N}\), where \(N\) is the resolution of the grid, i.e.\ the number of points in the \(X\) and \(Z\) directions of the Cartesian grid, the maximally distorted geodesic distance due to the projection, which occurs along the \(X=Z\) diagonal at the equator (or the \(X = -Z\) diagonal), is approximately \(\sqrt{12}\frac{2}{N}\) for large resolution \(N\). It is appropriate to take half of this distance, approximately \(3.5/N\), as the value of \(\varepsilon\) that allows cutting lines in the projection to remain continuous. In measurements of \(\Phi\) for a given \(\varepsilon\), the suggested grid size by this maximal distortion argument is much finer than is typically necessary for accurate area measurements of \(\tilde{E}\).

\section{Proof of polygonal cells along resonant branches in \(\alpha = \beta\) symmetry plane} \label{sec:non-ortho_angles}

For orthogonal rotation axes, the resonances along the \(\alpha=\beta\) symmetry line corresponds to polygonal tilings. In particular, for the \(j\)-th resonance, on the tongue extending from \(\alpha=\pi/(j+1)\), the cell with itinerary \(41^j\) is a \((2j+1)\)-gon. Itinerary here is given as a pathway a periodic point takes through the HS, each digit indicating which atom is visited and the superscript indicating repeated visits \cite{Smith2017}. The \((2j+1)\)-gon means that the rotation angle about the center of the cell with itinerary \(41^j\) is \(\theta_{41^j}=2k\pi/(2j+1)\), for some \(k\) coprime to \(2j+1\). Through direct computation of \(\theta_{41^j}\) it can be shown that that \(k=j\).

Varying the angle between the two rotation axes, \(\gamma\), a polygonal tiling can only remain if the cell with itinerary \(41^j\) remains a \((2j+1)\)-gon, i.e.\ the rotation generated by the itinerary \(41^j\) must be invariant under changes in \(\gamma\) and \(\alpha = \beta\). Essentially, we must find the level set such that \(\theta_{41^j}(\alpha,\gamma)=2j\pi/(2j+1)\).

In each atom, the map \(M_{\alpha,\beta,\gamma}\) can be expressed as the composition of two rotations:
\begin{align}
P_1: \, R_1 &= R_{\beta}^{\bm{u}_\gamma} R_{\alpha}^z \\
P_2: \, R_2 &= R_{\beta}^{\bm{u}_\gamma} R_{\alpha+\pi}^z \\
P_3: \, R_3 &= R_{\beta+\pi}^{\bm{u}_\gamma} R_{\alpha}^z \\
P_4: \, R_4 &= R_{\beta+\pi}^{\bm{u}_\gamma} R_{\alpha+\pi}^z,
\end{align}
where \(\bm{u}_\gamma=R_\gamma^y(\hat{\bm{z}})\) is the second rotation axis. Over a full itinerary, the net rotation is the composition of these atomic rotations. For example, for the itinerary \(41^2\) the net rotation is \(R_{41^2} = R_1 R_1 R_4\), noting that the rightmost rotation is performed first. We use the quaternion representation of rotations, i.e.\ rotation by \(\theta\) about the axis \(\bm{u}\) is
\begin{equation}
\bm{q}=\{q_1,q_2,q_3,q_4\}=\{\cos(\theta/2),\sin(\theta/2)\bm{u}\},
\end{equation}
and the composition of rotations is equivalent to products of quaternions. Note that rotation by \(n\theta\) about the axis \(\bm{u}\) is represented by \(\bm{q}^n\), and so
\begin{align}
\bm{q}^n &= \left\lbrace \cos \frac{n\theta}{2}, \sin \frac{n\theta}{2} \bm{u} \right\rbrace \\
 &= \left\lbrace T_n\left( \cos\frac{\theta}{2} \right), \sin\frac{\theta}{2} U_{n-1}\left( \cos\frac{\theta}{2} \right) \bm{u} \right\rbrace  \\
 &= \left\lbrace T_n(q_1), U_{n-1}(q_1) (q_2,q_3,q_4) \right\rbrace, \label{eq:multi-angle_quat}
\end{align}
where \(T_n(x)\) and \(U_n(x)\) are the \(n\)-th Chebyshev polynomials of the first and second kind, respectively. The Chebyshev polynomials are defined by the recurrence relations
\begin{align}
T_{n+1}(x) &= 2x T_n(x) - T_{n-1}(x) \label{eq:Chebyshev_recurrence_1} \\
U_{n+1}(x) &= 2x U_n(x) - U_{n-1}(x), \label{eq:Chebyshev_recurrence_2} 
\end{align}
with \(T_0(x)=U_0(x) = 1\), \(T_1(x)=x\), and \(U_1(x)=2x\), and satisfying
\begin{align}
T_n(\cos(\theta)) &= \cos(n\theta), \\
U_n(\cos(\theta)) &= \frac{\sin((n+1)\theta)}{\sin(\theta)}.
\end{align}

Let \(\bm{Q}^{(1)},\dots,\bm{Q}^{(4)}\) denote the quaternion representations of the atomic rotations \(R_1,\dots,R_4\) and define \(\sss_\theta \equiv \sin\theta/2\) and \(\ccc_\theta \equiv \cos\theta/2\) for brevity. For example,

\begin{align}
\bm{Q}^{(1)} &= \Bigl\lbrace  \ccc_\beta,\sss_\beta \sin\gamma, 0, \sss_\beta \cos\gamma \Bigr\rbrace \Bigl\lbrace \ccc_\beta,0 , 0, \sss_\alpha \Bigr\rbrace    \\
 &= \Bigl\lbrace  \ccc_\alpha \ccc_\beta - \cos\gamma \sss_\alpha \sss_\beta, \ccc_\alpha \sss_\beta \sin\gamma, \nonumber \\
 & \quad\quad\quad  - \sin\gamma \sss_\alpha \sss_\beta, \sss_\alpha \ccc_\beta + \cos\gamma \ccc_\alpha \sss_\beta \Bigr\rbrace. \label{eq:atom_1_quat}
\end{align}

We denote the quaternion representation of the itinerary \(41^j\) by \(\bm{q}^{(j)} = \{q_1^{(j)},q_2^{(j)},q_3^{(j)},q_4^{(j)}\}\), i.e.\ \(\bm{q}^{(j)} = \left(\bm{Q}^{(1)}\right)^j \bm{Q}^{(4)}\). We will show that when \(\alpha=\beta\), there is a relatively simple closed form expression for \(\bm{q}^{(j)}\), for all \(j\geq 1\). 

Substituting \(\beta=\alpha\) into Eq.~\ref{eq:atom_1_quat} yields
\begin{align}
\bm{Q}^{(1)}_1 &= \ccc_\alpha^2 - \cos\gamma \sss_\alpha^2  \nonumber\\
	&= 1 - 2\ccc^2_\gamma \sss_\alpha^2  \nonumber\\
	&= - T_2\left( \ccc_\gamma \sss_\alpha \right) , \\
\bm{Q}^{(1)}_2 &= \ccc_\alpha \sss_\alpha \sin\gamma, \\
\bm{Q}^{(1)}_3 &= - \sss^2_\alpha \sin\gamma, \\
\bm{Q}^{(1)}_4 &= (1+\cos \gamma)\sss_\alpha \ccc_\alpha  \nonumber \\ 
	&= 2 \ccc^2_\gamma \sss_\alpha \ccc_\alpha .
\end{align}
Therefore, from Eq.~\ref{eq:multi-angle_quat}, 
\begin{equation} \label{eq:atom_1_j}
 \left(\bm{Q}^{(1)}\right)^j = \left\lbrace T_j\left(\bm{Q}^{(1)}_1\right), U_{j-1}\left(\bm{Q}^{(1)}_1\right) \left(\bm{Q}^{(1)}_2,\bm{Q}^{(1)}_3,\bm{Q}^{(1)}_4\right) \right\rbrace.
\end{equation}
Noting that the Chebyshev polynomials satisfy the properties
\begin{align}
T_n\left(T_m(x)\right) &= T_{nm}(x), \\
U_n\left(T_m(x)\right) &= \frac{U_{(n+1)m-1}(x)}{U_{m-1}(x)}, \\
T_n(-x) &= (-1)^n T_n(x), \\
U_n(-x) &= (-1)^n U_n(x),
\end{align}  
it follows that
\begin{align}
T_j\left(\bm{Q}^{(1)}_1\right) &= T_j \left( - T_2\left( \ccc_\gamma \sss_\alpha \right) \right)  \nonumber \\
	&= (-1)^j T_{2j}\left( \ccc_\gamma \sss_\alpha \right), \\
U_{j-1}\left(\bm{Q}^{(1)}_1\right) &= U_{j-1}\left( - T_2\left( \ccc_\gamma \sss_\alpha \right) \right)  \nonumber\\
	&= (-1)^{j-1} \frac{U_{2j-1} \left( \ccc_\gamma \sss_\alpha \right)}{U_{1} \left( \ccc_\gamma \sss_\alpha \right)}  \nonumber\\
	&= (-1)^{j-1} \frac{U_{2j-1} \left( \ccc_\gamma \sss_\alpha \right)}{2 \ccc_\gamma \sss_\alpha }.
\end{align}
Substituting these into Eq.~\ref{eq:atom_1_j} and simplifying, it can be shown that
\begin{align}
\left(\bm{Q}^{(1)}\right)^j_1 &= T_j\left(\bm{Q}^{(1)}_1\right) \nonumber \\
	&= (-1)^j T_{2j}\left( \ccc_\gamma \sss_\alpha \right),  \\
\left(\bm{Q}^{(1)}\right)^j_2 &= U_{j-1}\left(\bm{Q}^{(1)}_1\right) \bm{Q}^{(1)}_2  \nonumber \\
	&= (-1)^{j-1} \ccc_\alpha \sss_\gamma U_{2j-1}\left( \ccc_\gamma \sss_\alpha \right)  \\
\left(\bm{Q}^{(1)}\right)^j_3 &= U_{j-1}\left(\bm{Q}^{(1)}_1\right) \bm{Q}^{(1)}_3 \nonumber \\
	&= (-1)^{j} \sss_\alpha \sss_\gamma U_{2j-1}\left( \ccc_\gamma \sss_\alpha \right)  \\
\left(\bm{Q}^{(1)}\right)^j_4 &= U_{j-1}\left(\bm{Q}^{(1)}_1\right) \bm{Q}^{(1)}_4 \nonumber \\
	&= (-1)^{j-1} \ccc_\alpha \ccc_\gamma U_{2j-1}\left( \ccc_\gamma \sss_\alpha \right).
\end{align}

Now, we can directly compute \(\bm{q}^{(j)} = \left(\bm{Q}^{(1)}\right)^j \bm{Q}^{(4)}\), where
\begin{align}
\bm{Q}^{(4)}_1 &= 1-2\ccc^2_\alpha \ccc^2_\gamma, \\
\bm{Q}^{(4)}_2 &= -2 \sss_\alpha \ccc_\alpha \sss_\gamma \ccc_\gamma, \\
\bm{Q}^{(4)}_3 &= -2\ccc^2_\alpha \sss_\gamma \ccc_\gamma, \\ 
\bm{Q}^{(4)}_4 &= -2 \sss_\alpha \ccc_\alpha \ccc^2_\gamma.
\end{align}
Simplifying the product \(\bm{q}^{(j)} = \left(\bm{Q}^{(1)}\right)^j \bm{Q}^{(4)}\), and letting  \(x=\cos(\gamma/2)\sin(\alpha/2)=\ccc_\gamma \sss_\alpha\), it can be shown that
\begin{align}
q_1^{(j)} &= (-1)^j \left[ T_{2(j+1)}(x) - 2(\ccc_\gamma^2 - 1) U_{2j}(x) \right], \label{eq:lem_1_eq_1}  \\
q_2^{(j)} &= (-1)^{j-1}\sss_\gamma \ccc_\alpha \left[ U_{2j+1}(x) + 2U_{2j-1}(x) \right],  \\
q_3^{(j)} &= (-1)^{j-1} \sss_\gamma \left[ 2\ccc_\gamma U_{2j}(x) \right. \nonumber \\
 &\qquad\qquad\qquad \left.-\sss_\alpha\left( U_{2j+1}(x) + 2U_{2j-1}(x) \right) \right], \\
q_4^{(j)} &= (-1)^{j-1} \ccc_\alpha  \ccc_\gamma U_{2j+1}(x). \label{eq:lem_1_eq_4}
\end{align}
%
%
Deriving Eqs.~\ref{eq:lem_1_eq_1}--\ref{eq:lem_1_eq_4} involves a lot of bookkeeping, but each step is not difficult. It is a matter of repeatedly exploiting standard trigonometric identities (e.g.\ double angle formulae), the recursive definitions of the Chebyshev polynomials (Eqs.~\ref{eq:Chebyshev_recurrence_1} and \ref{eq:Chebyshev_recurrence_2}), and the relation
\begin{equation}
U_n(x) = x U_{n-1}(x) + T_n(x),
\end{equation}
between Chebyshev polynomials of the first and second kind.

Using Eq.~\ref{eq:lem_1_eq_1}, the internal rotation angle, \(\theta_{41^j}\), generated by the cell with itinerary \(41^j\) can be calculated as
\begin{equation}
\theta_{41^j} = 2\arccos q_1^{(j)}.
\end{equation}

\begin{proposition} \label{prop:1}
For all \(j\geq 1\), if 
\begin{equation} \label{eq:non-ortho_resonances}
\ccc_\gamma \sss_\alpha=\sin\left(\frac{\pi}{2(2j+1)}\right)=\cos\left(\frac{j\pi}{2j+1}\right),
\end{equation}
then 
\begin{equation}
\theta_{41^j}(\alpha,\gamma)=\frac{2j\pi}{2j+1}.
\end{equation}
\end{proposition}

\begin{proof}
The aim is to show that if \(\ccc_\gamma \sss_\alpha = \cos(j\pi/(2j+1))\), then 
\begin{equation} \label{eq:target}
q_1^{(j)} = \cos(j\pi/(2j+1)),
\end{equation}
so that \(\theta_{41^j}=2j\pi/(2j+1)\).

Recall that \(T_n(x)\) and \(U_n(x)\) satisfy
\begin{align}
T_n(\cos(\theta)) &= \cos(n\theta), \\
U_n(\cos(\theta)) &= \frac{\sin((n+1)\theta)}{\sin(\theta)}.
\end{align}
Hence, substituting \(x=\ccc_\gamma \sss_\alpha = \cos(j\pi/(2j+1))\) into Eq.~\ref{eq:lem_1_eq_1} yields
\begin{widetext}
\begin{align}
q_1^{(j)} &= (-1)^j\left[ T_{2(j+1)}(\cos(j\pi/(2j+1))) - 2(\ccc_\gamma^2 - 1) U_{2j}(\cos(j\pi/(2j+1))) \right] \nonumber \\
 &= (-1)^j\left[ \cos\left( \frac{2j(j+1)\pi}{2j+1} \right) - 2(\ccc_\gamma^2 - 1) \frac{\sin(j\pi)}{\sin(j\pi/(2j+1))} \right] \nonumber \\
 &= (-1)^j\left[ \cos\left( j\pi + \frac{j\pi}{2j+1} \right) + 0 \right]  \nonumber\\
 &= (-1)^j\left[ (-1)^j \cos\left( \frac{j\pi}{2j+1} \right) \right]  \nonumber\\
 &= \cos\left( \frac{j\pi}{2j+1} \right),
\end{align}
\end{widetext}
as required.

\end{proof}

For example, for \(j=1\), when \(x=\ccc_\gamma \sss_\alpha = \cos(\pi/3)\),
\begin{align}
q_1^{(1)} &= -T_4(\cos(\pi/3)) + 2(\ccc_\gamma^2-1) U_2(\cos(\pi/3))   \nonumber\\ 
 &= -\cos(4\pi/3) + 2(\ccc_\gamma^2-1) \frac{\sin(\pi)}{\sin(\pi/3)}   \nonumber\\ 
 &=  \cos(\pi/3) + 0, \nonumber
\end{align}
and so \(\theta_{41}=2\arccos(q_1^{(1)}) = 2\pi/3\). This means the cell with itinerary \(41\) is a triangle.

Note that the above curves are not the only curves where \(\theta_{41^j}(\alpha,\gamma)=2j\pi/(2j+1)\). For example, with \(j = 2\), there are other protocols such that the \(41^2\) cell is pentagonal. However, the other curves do not result in polygonal tilings, or resonances. That is, the existence of some polygonal cells is a necessary, but not sufficient, condition for a complete polygonal tiling. 

It is also worth noting that when \(\gamma=\pi/2\), Eq.~\ref{eq:non-ortho_resonances} gives the angles \(\alpha\) corresponding to the resonances analytically, and in a simple, closed, explicit form. Previously, Smith et al.\ \cite{Smith2017} had only found an implicit equation for these resonances.


%

\end{document}